\documentclass{scrartcl}

\usepackage{lmodern}
\usepackage[T1]{fontenc}
\usepackage{microtype}
\usepackage[english]{babel}
\usepackage{csquotes}

\usepackage[backend=biber,style=alphabetic,url=false,giveninits=true,sorting=anyt]{biblatex}
\usepackage{amsmath,amssymb,amsthm}
\usepackage{mathtools,thmtools}
\usepackage{mathrsfs}

\usepackage{hyperref}
\hypersetup{
colorlinks=true,
linkcolor=blue!50!black,
citecolor=green!40!black,
urlcolor=blue!50!black
}
\usepackage[shortlabels]{enumitem}
\usepackage[nameinlink]{cleveref}
\crefname{equation}{}{}
\crefname{enumi}{}{}

\usepackage[svgnames]{xcolor}

\newcommand{\cQ}{\mathscr{Q}}
\newcommand{\cE}{\mathscr{E}}
\newcommand{\X}{\mathbf{X}}
\newcommand{\N}{\mathbb{N}}
\newcommand{\R}{\mathbb{R}}

\DeclareMathOperator{\lin}{lin}

\newcommand{\eps}{\varepsilon}
\newcommand{\genrel}{\mathrel{\mathcal R}}

\DeclarePairedDelimiter\abs{\lvert}{\rvert}
\DeclarePairedDelimiter\norm{\lVert}{\rVert}
\DeclarePairedDelimiterX{\pair}[2]{\langle}{\rangle}{#1, #2}

\DeclareMathOperator{\clspn}{\overline{span}}
\DeclareMathOperator{\argmin}{arg\,min}

\DeclareMathOperator{\dom}{dom}
\DeclareMathOperator{\supp}{supp}
\DeclareMathOperator{\cc}{cone}
\DeclareMathOperator{\cl}{cl}
\DeclareMathOperator{\clcc}{\overline{cone}}
\DeclareMathOperator{\relint}{ri}
\newcommand{\nop}{\mathrm{N}}
\DeclareMathOperator{\closedconv}{\overline{conv}}

\numberwithin{equation}{section}

\declaretheorem[name=Theorem,within=section]{thm}
\declaretheorem[name=Theorem]{mainthm}

\declaretheorem[name=Corollary,sibling=mainthm]{maincor}

\declaretheorem[name=Lemma,numberlike=thm]{lem}
\declaretheorem[name=Proposition,numberlike=thm]{prp}
\declaretheorem[name=Corollary,numberlike=thm]{cor}

\declaretheorem[name=Definition,numberlike=thm,style=definition]{dfn}

\declaretheorem[name=Remark,numberlike=thm,style=remark]{rmk}
\declaretheorem[name=Example,numberlike=thm,style=remark]{xmp}

\title{The quasi-convex Rockafellar theorem}
\author{Luigi De Pascale\thanks{Dipartimento di Matematica e Informatica, Università di Firenze, Firenze, Italy. Email: \texttt{luigi.depascale@unifi.it}}, Paul Pegon\thanks{CEREMADE, Université Paris--Dauphine, Université PSL, CNRS \& MOKAPLAN, Inria Paris, 75016 Paris, France. Email: \texttt{pegon@ceremade.dauphine.fr}}}
\date{\today}

\begin{document}
\maketitle
\begin{abstract}
	Rockafellar showed that every cyclically monotone operator is contained in the subdifferential of a convex function, and characterized the maximal ones. We establish the quasi-convex analogue of this theorem, with no assumption on the operator: in any separable Banach space, a setting we show to be sharp, every cyclically quasi-monotone operator is contained in the normal cone operator of a quasi-convex function. In finite dimension, we characterize maximal cyclically quasi-monotone operators and show that they are maximal quasi-monotone. Unlike in the convex case, the potential is not unique, and we build it from a preference relation obtained through successive preorders. As applications, we extend Afriat's theorem on revealed preferences in economics to infinite data sets, and obtain Kantorovich-type potentials for a supremal optimal transport problem.
\end{abstract}

{\small
\noindent\textit{Keywords:} Quasi-convex functions, cyclical quasi-monotonicity,
quasi-monotone operators, revealed preference, $L^\infty$ optimal transport,
cyclical monotonicity, generalized convexity, Afriat's theorem, normal cone,
utility representation.

\smallskip
\noindent\textit{MSC 2020:} Primary 26B25, 47H05; Secondary 49J52, 49Q22, 91B08, 91B16.
\par}

\hypersetup{
pdftitle={The quasi-convex Rockafellar theorem},
pdfkeywords={quasi-convex functions, cyclical quasi-monotonicity,
quasi-monotone operators, revealed preference, L-infinity optimal transport,
cyclical monotonicity, generalized convexity, Afriat's theorem,
normal cone, utility representation}
}

%

\tableofcontents

\section{Introduction}

A classical result of convex analysis, due to Rockafellar \cite{rockafellarCharacterizationSubdifferentialsConvex1966,rockafellarMaximalMonotonicitySubdifferential1970}, states that an operator is contained in the subdifferential of a convex function if and only if it is cyclically monotone, that the maximal cyclically monotone operators are exactly the subdifferentials of proper lower semicontinuous convex functions, and that these are moreover maximal monotone. Quasi-convex functions, i.e., functions with convex sublevel sets, are a natural generalization of convex functions when only the ordering of values matters, as for utility functions in economics, or for supremal functionals in the calculus of variations. Their first-order behavior is partly described by the normal cones to their sublevel sets, which provide the direction of steepest ascent but not the slope, and the natural counterpart of cyclical monotonicity is cyclical quasi-monotonicity.

Whether Rockafellar's theorem has a quasi-convex analogue has been investigated by several authors, in generalized convex analysis as well as in revealed preference theory \cite{daniilidisSubdifferentialsQuasiconvexPseudoconvex1999,bachirLowerSubdifferentiabilityIntegration2002,crouzeixConvexLevelSets2016,renyCharacterizationRationalizableConsumer2015,depascaleRockafellarTheoremCyclically2025}, but existing integration results require strong assumptions on the operator, such as regularity, nondegeneracy, or the existence of an orienting direction. In this paper, we prove the quasi-convex Rockafellar theorem in full generality: in any separable Banach space, a setting we show to be sharp, every cyclically quasi-monotone operator is contained in the normal cone operator of a lower semicontinuous quasi-convex function. In finite dimension, we also show that the maximal cyclically quasi-monotone operators are exactly the normal cone operators of a natural class of lower semicontinuous quasi-convex functions (see \Cref{main_thm}), and that these operators are moreover maximal quasi-monotone. The main difficulty is the lack of uniqueness: unlike in the convex case, a maximal operator does not determine such a function, even up to composition with an increasing function, so that a choice has to be made. Our strategy relies on the construction of four successive preorders associated with the operator, the last one being a lower semicontinuous quasi-convex preference relation from which the function is built.

As applications, we extend Afriat's theorem in revealed preference theory to possibly infinite data sets under an axiom weaker than the Generalized Axiom of Revealed Preference (GARP), which we call the Quasi Axiom of Revealed Preference (QARP), and, for the supremal optimal transport problem whose cost is the negative inner product $-\pair{x}{y}$, we obtain a nontrivial Kantorovich-type potential, which prescribes the direction of transport.

\paragraph{Setting and definitions}
We work in a real Banach space $\X$, with topological dual $\X^*$ consisting of all continuous linear forms on $\X$, and we denote by $\pair{x^*}{x}$ the duality pairing between any $x^* \in \X^*$ and $x \in \X$. A set-valued\footnote{We shall often drop the adjective \enquote{set-valued} for brevity.} operator $T \colon \X \multimap \X^*$ is a map from $\X$ to $2^{\X^*}$, which we identify with its graph, so that $T \subseteq S$ means $T(x) \subseteq S(x)$ for every $x \in \X$, in which case $S$ is called an \emph{extension} of $T$. An operator is \emph{maximal} in a given class (monotone, cyclically monotone, quasi-monotone, cyclically quasi-monotone operators, etc.) if it admits no extension in this class other than itself.

We recall that $T$ is \emph{cyclically monotone} if for every sequence of points $x_0, \ldots, x_N \in \X$ satisfying $x_N = x_0$, and every $x_0^* \in T(x_0), \ldots, x_{N-1}^* \in T(x_{N-1})$, the following inequality holds:
\begin{equation}\tag{CM}\label{cyclical_monotonicity}
	\sum_{i=0}^{N-1} \pair{x^*_i}{x_{i+1}-x_i} \leq 0.
\end{equation}
If this inequality is only required for $N = 2$, then $T$ is said to be \emph{monotone}. With these notions at hand, we may recall Rockafellar's theorem, in its complete form.

\begin{thm}[Rockafellar \cite{rockafellarCharacterizationSubdifferentialsConvex1966,rockafellarMaximalMonotonicitySubdifferential1970}]\label{Rockafellar_Theorem}
	Let $T \colon \X \multimap \X^*$ be a set-valued operator and let $\Gamma_0(\X)$ denote the set of proper lower semicontinuous convex functions $f \colon \X \to \R \cup \{+\infty\}$. The following statements hold.
	\begin{enumerate}
		\item \emph{(Existence)} $T$ is cyclically monotone if and only if $T \subseteq \partial f$ for some $f \in \Gamma_0(\X)$. Such an $f$ is called a convex potential of $T$.
		\item \emph{(Maximal cyclical monotonicity)} $T$ is maximal cyclically monotone if and only if $T = \partial f$ for some $f \in \Gamma_0(\X)$.
		\item \emph{(Uniqueness)} If $\partial f = \partial g$ with $f, g \in \Gamma_0(\X)$, then $f - g$ is constant. Thus, the convex potential of a maximal cyclically monotone operator is unique up to an additive constant.
		\item \emph{(Maximal monotonicity)} For every $f \in \Gamma_0(\X)$, the operator $\partial f$ is maximal monotone.
	\end{enumerate}
\end{thm}

To state our results, we recall the notions of quasi-convexity and cyclical quasi-monotonicity.

\begin{dfn}[Quasi-convexity, (cyclical) quasi-monotonicity]\label{dfn_qc_cqm}
	\leavevmode
	\begin{enumerate}[(i)]
		\item A function $f \colon \X \to \R \cup \{+\infty\}$ is \emph{quasi-convex} if its sublevel sets $\{x \in \X : f(x) \leq \ell\}$ are convex for every $\ell \in \R$.
		\item A set-valued operator $T \colon \X \multimap \X^*$ is \emph{cyclically quasi-monotone} if for every sequence of points $x_0, \ldots, x_N \in \X$ satisfying $x_N = x_0$, and every $x^*_0 \in T(x_0), \ldots, x^*_{N-1} \in T(x_{N-1})$, the following inequality holds:
		\begin{equation}\tag{CQM}\label{cyclical_quasi_monotonicity}
			\min_{0 \leq i < N} \pair{x^*_i}{x_{i+1}-x_i} \leq 0.
		\end{equation}
		If this inequality is only required for $N = 2$, then $T$ is said to be \emph{quasi-monotone}.
	\end{enumerate}
\end{dfn}

Cyclically quasi-monotone operators arise naturally from quasi-convex functions, through their normal cone operator. Denoting by $A^\circ \coloneqq \{x^* \in \X^* : \pair{x^*}{a} \leq 0 \text{ for every } a \in A\}$ the polar cone of a set $A \subseteq \X$, recall that the \emph{normal cone} to a set $C \subseteq \X$ at $x \in C$ is
\[
	\nop_C(x) \coloneqq (C - x)^\circ = \{x^* \in \X^* : \pair{x^*}{y-x} \leq 0 \text{ for every } y \in C\},
\]
which coincides with the subdifferential $\partial \chi_C(x)$ of the indicator function of $C$, with the convention $\nop_C(x) \coloneqq \emptyset$ if $x \notin C$. The \emph{normal cone operator} of a quasi-convex function $f \colon \X \to \R \cup \{+\infty\}$ is the operator $\nop f \colon \X \multimap \X^*$ defined by
\begin{equation}\label{normal_cone_operator_function}
	\forall x \in \X, \qquad \nop f(x) \coloneqq \nop_{\{f \leq f(x)\}}(x),
\end{equation}
with the convention $\nop f(x) \coloneqq \emptyset$ if $f(x) = +\infty$. The operator $\nop f$ is cyclically quasi-monotone, as a particular case of \Cref{convex_chain_operator_CQM}, which deals more generally with chains of sets (that is, families of sets totally ordered by inclusion).

\paragraph{Main result}

We establish the quasi-convex counterpart of Rockafellar's theorem: convexity is replaced by quasi-convexity, cyclical monotonicity by cyclical quasi-monotonicity, and the subdifferential operator $\partial$ by the normal cone operator $\nop$.

We denote by $\cQ_0(\X)$ the set of real-valued lower semicontinuous quasi-convex functions $f \colon \X \to \R$, and by $\cE_0(\X)$ the subset of functions in $\cQ_0(\X)$ that are \emph{essentially non-satiated}, meaning that every local minimizer is a strict global minimizer (see \Cref{dfn_essential_non_satiation}).

\begin{mainthm}\label{main_thm}
	Let $T \colon \X \multimap \X^*$ be a set-valued operator on a separable Banach space $\X$. The following statement holds.
	\begin{enumerate}[ref=(\arabic*)]
		\item\label{main_thm_existence} \emph{(Existence)} $T$ is cyclically quasi-monotone if and only if $T \subseteq \nop f$ for some $f \in \cQ_0(\X)$. Such an $f$ is called a quasi-convex potential of $T$.
	\end{enumerate}
	Moreover, if $\X$ is finite-dimensional, the following statements also hold.
	\begin{enumerate}[resume, ref=(\arabic*)]
		\item\label{main_thm_maximal_CQM} \emph{(Maximal cyclical quasi-monotonicity)} $T$ is maximal cyclically quasi-monotone if and only if $T = \nop f$ for some $f \in \cE_0(\X)$.
		\item\label{main_thm_nonuniqueness} \emph{(Nonuniqueness)} In general, the quasi-convex potential of a maximal cyclically quasi-monotone operator is not unique, even up to composition with an increasing function.
		\item\label{main_thm_maximal_QM} \emph{(Maximal quasi-monotonicity)} For every $f \in \cE_0(\X)$, the operator $\nop f$ is maximal quasi-monotone.
	\end{enumerate}
\end{mainthm}

\begin{rmk}[Necessity of separability]\label{rmk_separability}
	The separability assumption, which is absent from Rockafellar's theorem, is necessary: there exists a cyclically quasi-monotone operator on a nonseparable Banach space that admits no potential (\Cref{prp_nonseparable_counterexample}). In a more general setting, that of smoothable Banach spaces (which include all reflexive and all separable ones), every cyclically quasi-monotone operator is contained in the normal cone operator $\nop_\preceq$ of some lower semicontinuous quasi-convex preference relation $\preceq$ (\Cref{thm_chain}, see \labelcref{normal_cone_operator_preference} for the definition of $\nop_\preceq$), but separability is needed to represent such a relation by a real-valued function (\Cref{cor_rader_separable}).
\end{rmk}

\begin{rmk}[Nonuniqueness]\label{rmk_nonuniqueness}
	The failure of uniqueness is severe: even the family of sublevel sets of a potential is not determined by the operator (\Cref{example_non_uniqueness}), as already observed by Crouzeix, Eberhard and Ralph for pseudoconvex functions \cite[Example~6.2]{crouzeixConvexLevelSets2016}. A choice is literally made in the proof: it relies on the axiom of choice through the Arrow--Hansson theorem (\Cref{pre-order_extension}), a generalization of Szpilrajn's theorem, which extends any partial order to a total order, in a non-trivial way.
\end{rmk}

\paragraph{Application to revealed preferences and OT}
A \emph{data set} is a set $\mathcal D \subseteq \R^d \times \R^d$ of pairs $(x,p)$, where $x$ is a chosen bundle of goods and $p$ the corresponding price vector. A utility function $u$ \emph{rationalizes} $\mathcal D$ if $u(y) \leq u(x)$ for every $(x,p) \in \mathcal D$ and every $y$ such that $\pair{p}{y} \leq \pair{p}{x}$. Afriat's theorem reads as follows.

\begin{thm}[Afriat \cite{afriatConstructionUtilityFunctions1967}, Varian \cite{varianNonparametricApproachDemand1982}]\label{thm_afriat}
	Let $\mathcal D = \{(x_i,p_i)\}_{1 \leq i \leq n}$ be a finite data set with $p_i \in \mathopen]0,+\infty\mathclose[^d$. The following assertions are equivalent.
	\begin{enumerate}[(i)]
		\item $\mathcal D$ is rationalized by a locally non-satiated\footnote{A utility function is locally non-satiated if it admits no local maximizer.} utility function, which can moreover be chosen continuous and concave.
		\item $\mathcal D$ satisfies GARP: for every cycle $i_0, \ldots, i_N = i_0$ such that $\pair{p_{i_k}}{x_{i_{k+1}}} \leq \pair{p_{i_k}}{x_{i_k}}$ for every $k < N$, all these inequalities are equalities.
	\end{enumerate}
\end{thm}

Setting $T(x) \coloneqq \{-p : (x,p) \in \mathcal D\}$, GARP corresponds to the cyclical pseudo-monotonicity of $T$ \cite{karamardianComplementarityProblemsCones1976,daniilidisSubdifferentialsQuasiconvexPseudoconvex1999}, while its cyclical quasi-monotonicity corresponds to the weaker requirement that there is no cycle of strict direct revealed preferences (i.e., in which all the inequalities are strict). \Cref{main_thm} then yields the following generalization of Afriat's theorem, which is discussed and proved in \Cref{sec_revealed_preferences}.

\begin{maincor}[Generalized Afriat theorem]\label{cor_afriat}
	Let $\X$ be a separable Banach space and $\mathcal D \subseteq \X \times \X^*$ be an arbitrary data set. The following assertions are equivalent.
	\begin{enumerate}[(i)]
		\item $\mathcal D$ is \emph{rationalized in the strict sense} by a utility function $u \colon \X \to \R$, i.e., strictly cheaper bundles are strictly worse:
		\[\forall (x,p) \in \mathcal D,\ \forall y \in \X, \qquad \pair{p}{y} < \pair{p}{x} \implies u(y) < u(x).\]
		Moreover, $u$ can be chosen upper semicontinuous and quasi-concave, and also essentially non-satiated (meaning that every local maximizer of $u$ is a strict global maximizer) if $\X$ is finite-dimensional.
		\item $\mathcal D$ satisfies the following weakening of GARP, which we call the \emph{Quasi Axiom of Revealed Preference} (QARP): there is no cycle of strict direct revealed preferences, i.e., no $(x_0,p_0), \ldots, (x_{N-1},p_{N-1}) \in \mathcal D$ such that, setting $x_N \coloneqq x_0$, $\pair{p_k}{x_{k+1}} < \pair{p_k}{x_k}$ for every $k < N$.
	\end{enumerate}
\end{maincor}

Compared with Afriat's theorem (\Cref{thm_afriat}), the data set may be infinite and lie in a separable Banach space, the price vectors may vanish and have components of any sign, and GARP is replaced by the weaker axiom QARP. In return, the utility function only rationalizes the data in the strict sense, it is quasi-concave and upper semicontinuous instead of concave and continuous, and, in finite dimension, essentially non-satiated instead of locally non-satiated, which is unavoidable in this generality. Infinite data sets in $\R^m_+$, with prices in $\R^m_+ \setminus \{0\}$, were already treated by Reny \cite{renyCharacterizationRationalizableConsumer2015}, who showed that GARP characterizes the data sets rationalized by a quasi-concave utility function, in the classical sense and in the strict sense simultaneously. Our result thus extends his to infinite dimension, arbitrary prices and QARP, at the cost of rationalization in the strict sense only. As in Afriat's theorem, and despite what the adjective \enquote{revealed} suggests, our result only provides the \emph{existence} of a utility function rationalizing the data, which is neither unique nor canonically determined by $\mathcal D$. We refer to \Cref{sec_revealed_preferences} for details, in particular for a discussion of the two notions of rationalization.

Our second application concerns the supremal (or $L^\infty$) optimal transport problem with cost $-\pair{x}{y}$, which consists in minimizing the maximal, or worst-case, cost of a transport plan, instead of its average cost as in classical optimal transport. In the classical problem, typically with the quadratic cost $\abs{x-y}^2$ or equivalently $-\pair{x}{y}$, the functional to minimize is linear, hence convex, in the transport plan, whereas the supremal functional is only quasi-convex. Classical optimal plans are characterized by the cyclical monotonicity of their supports, so that Rockafellar's theorem yields a convex potential $\varphi$ which prescribes the transport, in the sense that the support is contained in the graph of $\partial\varphi$. This is the Brenier potential (or Kantorovich potential in this context), and the optimal plan is unique when the source measure is absolutely continuous \cite{brenierPolarFactorizationMonotone1991}. In the supremal problem, optimal plans are far from unique and convex duality is not available. However, a special subclass of solutions, called restrictable solutions, is characterized by the $\infty$-cyclical monotonicity of the support, and these solutions are induced by transport maps under suitable assumptions \cite{championInftyWassersteinDistance2008,jylhaInftyOptimalTransport2015}. A duality theory was proposed by Barron, Bocea and Jensen \cite{barronDualityInftyOptimal2017}, but its dual problem always admits trivial solutions, and the search for nontrivial Kantorovich-type potentials is still ongoing. Our result partly answers this question (\Cref{thm_supremal_OT}): for the cost $-\pair{x}{y}$, restrictable solutions exist, and the support of every restrictable solution $\gamma$ is cyclically quasi-monotone, so that it is contained in the graph of $\nop f$ for some quasi-convex function $f$, which thus prescribes the direction of transport. We refer to \Cref{sec_supremal_OT} for details.

\paragraph{Related works}
The integration of generalized monotone operators by generalized convex functions has a long history, in generalized convex analysis as well as in revealed preference theory. Bachir, Daniilidis and Penot integrated operators through the lower subdifferential of Plastria \cite{bachirLowerSubdifferentiabilityIntegration2002}. Crouzeix and coauthors integrated pseudomonotone maps in connection with the revealed preference problem \cite{crouzeixIntegrabilityPseudomonotoneDifferentiable2005,crouzeixIntegrationPseudomonotoneMaps2011,crouzeixGeometricalInsightPseudoconvexity2010}, and constructed pseudoconvex potentials of maximal cyclically pseudomonotone maps, which are not unique even after normalization \cite{crouzeixConstructiveProofExistence2012,crouzeixConvexLevelSets2016}. Reny extended Afriat's theorem to arbitrary data sets in $\R^m_+$, with prices in $\R^m_+ \setminus \{0\}$ \cite{renyCharacterizationRationalizableConsumer2015}. All these results rely on strong assumptions: regularity (differentiability, lower subdifferentiability), nondegeneracy (pseudomonotonicity, or GARP), or structure. For instance, \cite{crouzeixConvexLevelSets2016,renyCharacterizationRationalizableConsumer2015} assume that all normal vectors, or prices, lie in a fixed cone admitting an orienting direction $e$, i.e., $\pair{x^*}{e} > 0$ for all of them (the positive orthant with $e = (1,\ldots,1)$ in Reny's setting). In contrast, \Cref{main_thm} requires none of these assumptions, thus also extending the regular and non-vanishing case of \cite{depascaleRockafellarTheoremCyclically2025}: it only uses cyclical quasi-monotonicity, which is weaker than GARP, in any separable Banach space, and separability cannot be dropped (\Cref{rmk_separability}). Closely related questions are the characterization of quasi-convexity through subdifferentials or normal cones to sublevel sets \cite{ausselSubdifferentialCharacterizationQuasiconvexity1994,ausselNormalCharacterizationMain2000,ausselNormalConesSublevel2001}, and the maximality of quasi-monotone operators \cite{buenoMaximalityQuasimonotoneOperators2019}, whose criterion underlies \Cref{sec_maximality_normal_cone}.

\paragraph{Structure of the proof and organization of the paper}
The paper is organized as follows. In \Cref{sec_preliminaries}, we recall the necessary background on preference relations, show that separability is necessary for the existence of potentials, and introduce the preorders canonically associated with a cyclically quasi-monotone operator. \Cref{sec_existence} is devoted to maximal cyclically quasi-monotone operators and to the existence of potentials: we construct a lower semicontinuous quasi-convex preference relation $\preceq$ such that $T \subseteq \nop_\preceq$ (\Cref{thm_chain}), from which we deduce item~\labelcref{main_thm_existence} of \Cref{main_thm}. In \Cref{sec_maximal}, we characterize maximal cyclically quasi-monotone operators in finite dimension, give an example of nonuniqueness of the potential, and complete the proof of \Cref{main_thm}. Finally, \Cref{sec_applications} presents the applications to revealed preferences and supremal optimal transport.

\section{Preference relations and cyclically quasi-monotone operators}\label{sec_preliminaries}

In this section, we fix the vocabulary on preference relations and their representation by functions (\Cref{sec_preference_relations}), define the normal cone operators of chains of sets and of preference relations (\Cref{sec_normal_cone_operators}), show that separability is needed to find a potential (\Cref{sec_necessity_separability}), and introduce the preorders canonically associated with a cyclically quasi-monotone operator (\Cref{sec_canonical_preorders}).

It is convenient to rephrase cyclical quasi-monotonicity in terms of ascending paths. Given an operator $T \colon \X \multimap \X^*$, a \emph{$T$-ascending path from $x$ to $y$} is a sequence $(x_0, \ldots, x_N)$, with $N \geq 1$, such that $x_0 = x$, $x_N = y$ and, for every $i \in \{0, \ldots, N-1\}$, there exists $x_i^* \in T(x_i)$ satisfying $\pair{x_i^*}{x_{i+1}-x_i} > 0$. We call $(x^*_0, \ldots, x^*_{N-1})$ \emph{$T$-ascending directions} for this path, and when $x = y$, we call this path a \emph{$T$-ascending cycle}, which then necessarily has $N \geq 2$. With this notion, cyclical quasi-monotonicity can be expressed as follows:
\begin{equation}\label{CQM_no_ascending_cycle}
	T \text{ is cyclically quasi-monotone} \iff \text{there is no $T$-ascending cycle}.
\end{equation}

\subsection{Preference relations and their representation}\label{sec_preference_relations}

A binary relation $\mathcal R$ on a set $X$ is \emph{reflexive} (resp.\ \emph{irreflexive}) if $x \genrel x$ holds (resp.\ fails) for every $x \in X$, \emph{transitive} if $x \genrel y$ and $y \genrel z$ imply $x \genrel z$, and \emph{total} if $x \genrel y$ or $y \genrel x$ for every $x, y \in X$. A \emph{preorder} is a reflexive and transitive relation and it is a \emph{(partial) order} if moreover $x \genrel y$ and $y \genrel x$ imply $x = y$. A \emph{strict preorder} is an irreflexive and transitive relation, and a \emph{preference relation} is a total preorder.

A preference relation is entirely described by the chain of its lower contour sets. Indeed, given a preference relation $\preceq$ on $\X$, the family $\mathcal C_\preceq \coloneqq (C_\preceq(x))_{x \in \X}$  of its lower contour sets, defined by
\[
	C_\preceq(x) \coloneqq \{w \in \X : w \preceq x\} \qquad (x \in \X)
\]
is a chain, and the properties $x \in C_\preceq(x)$ and ($w \in C_\preceq(x) \implies C_\preceq(w) \subseteq C_\preceq(x)$) are nothing but the reflexivity and transitivity of $\preceq$. Conversely, $\preceq$ is recovered by $x \preceq y \iff x \in C_\preceq(y)$. We say that $\preceq$ is \emph{quasi-convex} (resp.\ \emph{lower semicontinuous}) if all its lower contour sets are convex (resp.\ closed).

Functions provide a third description, which is slightly more restrictive. A function $f \colon \X \to \R$ \emph{represents} $\preceq$ if $x \preceq y \iff f(x) \leq f(y)$, in which case $C_\preceq(x) = \{f \leq f(x)\}$ for every $x$, and $f$ is unique up to composition by a strictly increasing function. Every function induces a preference relation in this way, but the converse requires a countability assumption: a preference relation need not be representable by a real-valued function (see \Cref{rmk_countability} below), while the following classical result of Rader ensures that lower semicontinuous ones are, in second-countable spaces. We refer to \cite[Section~2.2]{depascaleRockafellarTheoremCyclically2025} for more details on these three points of view.

\begin{thm}[Rader]\label{thm_rader}
	For every lower semicontinuous preference relation $\preceq$ on a second-countable topological space $X$, there exists a lower semicontinuous function $f \colon X \to [0,1]$ such that
	\begin{equation}\label{potential_preference_relation_Rader}
		x \preceq y \iff f(x) \leq f(y)
	\end{equation}
	for every $x,y \in X$.
\end{thm}
\begin{rmk}[On the proof of Rader's theorem]
	The classical proof decomposes into two steps: first build a general function $f$ that satisfies \labelcref{potential_preference_relation_Rader}, then modify it to make it lower semicontinuous. For the first step, consider a countable basis $(O_n)_{n\in\N^*}$ of open sets of $X$, and set for every closed subset $C \subseteq X$,
	\[g(C) = \sum_{n\geq 1, O_n \subseteq X\setminus C} 2^{-n} \qquad\text{and}\qquad f(x) = 1-g(C_\preceq(x)).\]

	The second step is not easy, and relies on Debreu's Open Gap Lemma. In \cite{depascaleRockafellarTheoremCyclically2025}, we followed the same two-step strategy in the metric setting: for the first step, we gave an explicit function in the case of quasi-convex preference relations, and for the second step, we used the \emph{sublevel-stable lower semicontinuous envelope}, which we introduced there. This envelope, which is defined by
	\[\hat f(x) = \inf \left\{ \liminf_{n\to +\infty} f(x_n') : x_n' \to x'\text{ and } f(x') = f(x)\right\},\]
is lower semicontinuous and still represents the same relation. This gives another proof of Rader's theorem, which is less general as it is restricted to separable metric spaces, but whose construction is more explicit.
\end{rmk}

Since a function representing a quasi-convex preference relation is itself quasi-convex, we get the following corollary.

\begin{cor}\label{cor_rader_separable}
	For every lower semicontinuous and quasi-convex preference relation $\preceq$ on a separable Banach space $\X$, there exists a quasi-convex lower semicontinuous function $f \colon \X \to [0,1]$ representing $\preceq$, meaning that
	\[\forall x,y \in \X, \qquad x \preceq y \iff f(x) \leq f(y).\]
\end{cor}

\begin{rmk}[Necessity of a countability hypothesis]\label{rmk_countability}
	Notice that if \labelcref{potential_preference_relation_Rader} holds, then the chain $\mathcal C_\preceq$ admits an order-preserving indexation by a subset of $\R$, given by $C_\preceq(x) \mapsto f(x)$. It is therefore necessary for $\mathcal C_\preceq$ to have a cardinality at most that of $\R$. Even more than that, it is known that admitting such an indexation is equivalent to being \emph{order-separable}, which in our context means that there exists a countable subfamily $(C_n)_{n\in \N} \subseteq \mathcal C_\preceq$ such that for every $C, C' \in \mathcal C_\preceq$ that satisfy $C \subsetneq C'$, there exists $n$ for which $C \subseteq C_n \subseteq C'$. The second countability assumption on $X$ guarantees that any chain of closed sets is order-separable.
\end{rmk}

\subsection{Normal cone operators}\label{sec_normal_cone_operators}

We now extend the normal cone operator \labelcref{normal_cone_operator_function} to chains of sets. Given a chain $\mathcal C$ of subsets of $\X$, we define its \emph{normal cone operator} $\nop \mathcal C$ by
\begin{equation}\label{normal_cone_operator_chain}
	\forall x \in \X, \qquad \nop \mathcal C(x) \coloneqq \bigcup_{C \in \mathcal C : x \in C} \nop_C(x).
\end{equation}
This covers the two particular cases used in the paper: preference relations and quasi-convex functions. If $\preceq$ is a preference relation, we write $\nop_\preceq \coloneqq \nop\mathcal C_\preceq$. Since $C_\preceq(x)$ is the smallest element of $\mathcal C_\preceq$ containing $x$, and the normal cone at $x$ decreases as the set increases, we have
\begin{equation}\label{normal_cone_operator_preference}
	\forall x \in \X, \qquad \nop_\preceq(x) = \nop_{C_\preceq(x)}(x).
\end{equation}
Likewise, if $f \colon \X \to \R \cup \{+\infty\}$ is quasi-convex, the normal cone operator of the chain of its sublevel sets is $\nop f$, as defined in \labelcref{normal_cone_operator_function}. In particular, $\nop_\preceq = \nop f$ whenever $f$ represents $\preceq$.

\begin{prp}\label{convex_chain_operator_CQM}
	Let $\mathcal C$ be a chain of subsets of $\X$. Then $\nop\mathcal C$, defined in \labelcref{normal_cone_operator_chain}, is cyclically quasi-monotone.
\end{prp}

\begin{proof}
	By contradiction, assume that there exists an $\nop\mathcal C$-ascending cycle $x_0, \ldots, x_N = x_0$ with associated $\nop\mathcal C$-ascending directions $x_0^*, \ldots, x^*_{N-1}$:
	\begin{equation}\label{ascending_path_convex_chain}
		\forall j < N, \quad \pair{x^*_j}{x_{j+1}-x_j} > 0.
	\end{equation}
	By definition of $\nop\mathcal C$, for every $j < N$, $x^*_j \in \nop_{C_j}(x_j)$ for some $C_j \in \mathcal C$, which implies that $x_j \in C_j$ (recall that $\nop_C(x) = \emptyset$ when $x \notin C$). From \labelcref{ascending_path_convex_chain} we then deduce that $x_{j+1} \notin C_j$ for every $j < N$, and since $x_{j+1} \in C_{j+1}$ (with $C_N \coloneqq C_0$) and $\mathcal C$ is a chain, we obtain
	\[C_0 \subsetneq C_1 \subsetneq \cdots \subsetneq C_{N-1} \subsetneq C_0,\]
	a contradiction.
\end{proof}

Allowing chains of nonconvex sets does not bring more generality, since $(A-x)^\circ = (\closedconv A - x)^\circ$ for every $x \in A$, and the closed convex hulls of the elements of a chain still form a chain. This is why, when looking for potentials, we may restrict to chains of closed convex sets, or lower semicontinuous quasi-convex functions.

For preference relations, we will use the following regularization, which replaces the lower contour sets by their closures while keeping a transitive relation.

\begin{lem}[Lower semicontinuous regularization]\label{lsc_regularization}
	Let $\preceq$ be a preference relation on $\X$, and define the relation $\overline{\preceq}$ on $\X$ by
	\[x \mathrel{\overline{\preceq}} y \iff x \text{ belongs to every set } \cl C_\preceq(x'),\ x' \in \X, \text{ that contains } y.\]
	Then $\overline{\preceq}$ is a lower semicontinuous preference relation, which is quasi-convex if $\preceq$ is quasi-convex, and $\nop_\preceq \subseteq \nop_{\overline{\preceq}}$.
\end{lem}
\begin{proof}
	The relation $\overline{\preceq}$ is reflexive and transitive, and for every $y \in \X$, its lower contour set $C_{\overline{\preceq}}(y)$ is the intersection of the elements of $\mathcal K \coloneqq \{\cl C_\preceq(x') : x' \in \X\}$ that contain $y$. The family $\mathcal K$ is a chain, as the family of the closures of the elements of the chain $\mathcal C_\preceq$, and intersections of subfamilies of a chain form a chain (if $\bigcap \mathcal A \not\subseteq \bigcap \mathcal B$, some $B \in \mathcal B$ is contained in every $A \in \mathcal A$, so that $\bigcap \mathcal B \subseteq \bigcap \mathcal A$). Therefore the lower contour sets of $\overline{\preceq}$ form a chain, and since $x \in C_{\overline{\preceq}}(x)$ for every $x$, the relation $\overline{\preceq}$ is total. These sets are closed, and convex if $\preceq$ is quasi-convex, since the closure of a convex set is convex. Finally, $y \in C_{\overline{\preceq}}(y) \subseteq \cl C_\preceq(y)$ and $(A-y)^\circ = (\cl A - y)^\circ$ for every set $A$, so that, by \labelcref{normal_cone_operator_preference},
	\[\nop_\preceq(y) = \nop_{\cl C_\preceq(y)}(y) \subseteq \nop_{C_{\overline{\preceq}}(y)}(y) = \nop_{\overline{\preceq}}(y). \qedhere\]
\end{proof}

\subsection{Necessity of separability}\label{sec_necessity_separability}

By \Cref{rmk_countability}, a lower semicontinuous quasi-convex preference relation on a nonseparable space need not be representable by a function. This does not immediately provide a cyclically quasi-monotone operator without potential, since such an operator can coincide with the normal cone operator $\nop_{\preceq}$ of several preference relations $\preceq$, one of which may be representable by a function. Showing that separability is necessary in \Cref{main_thm} therefore requires a finer construction. It is inspired by the following example of a chain of subsets of $\R$ that cannot be indexed, in an order-preserving way, by $\R$. Consider the family $\mathcal F$ made of $A_t = \{u : u < t\}$ and $B_t = \{u : u \leq t\}$ for every $t \in \R$. If there were a function $\phi \colon \mathcal F \to \R$ strictly increasing for inclusion, then for each $s < t$ we would have $\phi(B_s) < \phi(A_t) < \phi(B_t)$, so that $\{(\phi(A_t), \phi(B_t))\}_{t\in \R}$ would form an uncountable family of disjoint nonempty open intervals of $\R$, which is impossible. In the following proposition, we transpose this example to a chain of closed vector subspaces, by assigning to each $t$ a copy of $\R$ and taking the direct sum.

\begin{prp}\label{prp_nonseparable_counterexample}
	There exists a cyclically quasi-monotone operator $T \colon \X \multimap \X^*$ on some nonseparable Banach space $\X$ such that $T \not\subseteq \nop f$ for every quasi-convex function $f \colon \X \to \R$.
\end{prp}

\begin{proof}
	We consider the Hilbert space $\X = \ell^2(\R,\R)$ consisting of functions $x$ which are square-summable, i.e., $\sum_{t\in \R} \abs{x(t)}^2 < +\infty$ (or equivalently, square-integrable for the counting measure), equipped with the scalar product $\pair{x}{y} = \sum_t x(t) y(t)$. We denote by $(e_s)_{s\in \R}$ the canonical orthonormal basis and we set for every $t \in \R$,
	\[V_t = \clspn \{ e_s : s < t\}, \qquad W_t = \clspn\{e_s : s \leq t\}.\]
	These closed subspaces form a chain $\mathcal C \coloneqq \{V_t\}_{t\in \R} \cup \{W_t\}_{t\in \R}$, and we set $T = \nop \mathcal C$, which is cyclically quasi-monotone by \Cref{convex_chain_operator_CQM}. We proceed by contradiction, assuming that $T \subseteq \nop f$ for some quasi-convex function $f \colon \X \to \R$.

	First of all, the normal cones of the elements of $\mathcal C$ are explicit, since $\nop_M(x) = M^\perp$ for any closed subspace $M$ and any $x \in M$. Let us compute $T(e_t)$. Notice that the subspaces in $\mathcal C$ that contain $e_t$ are exactly the $W_r$'s with $r \geq t$ and the $V_r$'s with $r > t$, and $W_t$ is the smallest one. This implies
	\[T(e_t) = \nop\mathcal C(e_t) = \nop_{W_t}(e_t) = W_t^\perp.\]
	Now set $v_t \coloneqq \sum_{n\geq 1} e^{-n} e_{t-1/n} \in V_t$ and compute $T(v_t)$. The subspaces in $\mathcal C$ that contain $v_t$ are all $W_r$'s for $r \geq t$ and all $V_r$'s for $r \geq t$, and $V_t$ is the smallest one. We get
	\[T(v_t) = \nop\mathcal C(v_t) = \nop_{V_t}(v_t) = V_t^\perp.\]

	Setting $V'_t \coloneqq \{f \leq f(v_t)\}$ and $W'_t \coloneqq \{f \leq f(e_t)\}$, the inclusion $T(v_t) \subseteq \nop f(v_t) = (V'_t - v_t)^\circ$ yields
	\[V'_t - v_t \subseteq T(v_t)^\circ = V_t^{\perp\circ} = V_t^{\perp\perp} = V_t,\]
	thus $V'_t \subseteq V_t$ since $v_t \in V_t$, and similarly $W'_t \subseteq W_t$.

	Since $e_t \in W'_t$ and $e_t \notin V_t \supseteq V'_t$, and since the sublevel sets of $f$ are nested, we get $V'_t \subsetneq W'_t$. As a consequence $f(v_t) < f(e_t)$ and
	\[I_t \coloneqq (f(v_t), f(e_t))\]
	is a nonempty open interval of $\R$. If $s < t$, then $v_t \in V'_t$ but $v_t \notin W_s \supseteq W'_s$. Again, since the sublevel sets of $f$ are nested, we get $W'_s \subsetneq V'_t$, thus $f(e_s) \leq f(v_t)$. Therefore $\sup I_s \leq \inf I_t$. We conclude that $\{I_s\}_{s\in \R}$ is an uncountable disjoint family of open intervals of $\R$: this is impossible, hence no potential $f$ for $T$ may exist.
\end{proof}

\subsection{The canonical preorders}\label{sec_canonical_preorders}

Let $T \colon \X \multimap \X^*$ be an operator. Following \cite{depascaleRockafellarTheoremCyclically2025}, we define two binary relations on $\X$ by
\begin{align}
	x \prec_T y & \iff \text{there exists a $T$-ascending path from $x$ to $y$},\label{def_strict_pre-order}\\
	\shortintertext{and}
	x \preceq_T y & \iff \{\cdot \prec_T x\} \subseteq \{\cdot \prec_T y\}.\label{def_pre-order}
\end{align}
We refer to \cite[Remark~2.15]{depascaleRockafellarTheoremCyclically2025} for a discussion of this choice of definition of $\preceq_T$, which favors transitivity over totality. The following proposition is stated in $\R^d$ in \cite{depascaleRockafellarTheoremCyclically2025}, but its proof holds verbatim in $\X$.

\begin{prp}[{\cite[Proposition~2.14]{depascaleRockafellarTheoremCyclically2025}}]\label{transitivity_properties}
	Let $T \colon \X \multimap \X^*$ be a cyclically quasi-monotone operator.
	\begin{enumerate}[(i)]
		\item\label{item_strict_pre-order} The relation $\prec_T$ is a strict preorder.
		\item\label{item_pre-order} The relation $\preceq_T$ is a preorder.
		\item\label{item_strict_implies_large} For every $x,y \in \X$,
		\begin{equation}\label{strict_implies_large}
			x \prec_T y \implies x \preceq_T y.
		\end{equation}
		\item\label{item_mixed_transitivity} The following mixed transitivity property holds: for every $x,y,z \in \X$,
		\begin{equation}\label{strlarstr}
			x \prec_T y \text{ and } y \preceq_T z \implies x \prec_T z.
		\end{equation}
	\end{enumerate}
\end{prp}

\begin{cor}\label{strict_part}
	Let $T \colon \X \multimap \X^*$ be a cyclically quasi-monotone operator. Then $\prec_T$ is contained in the strict part of $\preceq_T$, i.e., the relation defined by ($x \preceq_T y$ and $y \not\preceq_T x$).
\end{cor}
\begin{proof}
	If $x \prec_T y$, then $x \preceq_T y$ by \Cref{transitivity_properties} \labelcref{item_strict_implies_large}, but $y \not\preceq_T x$, since otherwise \Cref{transitivity_properties} \labelcref{item_mixed_transitivity} would give $x \prec_T x$, contradicting \Cref{transitivity_properties} \labelcref{item_strict_pre-order}.
\end{proof}

For every $x \in \X$, we denote by
\[S^T(x) \coloneqq \{\cdot \prec_T x\} \quad\text{and}\quad C^T(x) \coloneqq \{\cdot \preceq_T x\}\]
the strict lower contour set of $x$ for $\prec_T$ and its lower contour set for $\preceq_T$ (i.e., $C^T(x) = C_{\preceq_T}(x)$ with the notation of \Cref{sec_preference_relations}), so that $S^T(x) \subseteq C^T(x)$ by \Cref{transitivity_properties} \labelcref{item_strict_implies_large}. While $S^T(x)$ has no reason to be convex in general, the sets $C^T(x)$ are always closed and convex, as stated in the following proposition, whose proof in \cite{depascaleRockafellarTheoremCyclically2025} also holds verbatim in $\X$.

\begin{prp}[{\cite[Proposition~2.17]{depascaleRockafellarTheoremCyclically2025}}]\label{contour_sets_closed_convex}
	Let $T \colon \X \multimap \X^*$ be a cyclically quasi-monotone operator. For every $x \in \X$, the set $C^T(x)$ is closed and convex.
\end{prp}

The relation $\preceq_T$ is our first candidate preference relation. In \cite[Proposition~4.4]{depascaleRockafellarTheoremCyclically2025}, we showed that it is indeed total, hence a preference relation, when $\X = \R^d$ and $T$ is single-valued, locally Lipschitz and never vanishes, but it is not total in general, even when $T$ is maximal. For instance, $T \coloneqq \nop f$ with $f \coloneqq \abs{\cdot}$ on $\R$ is maximal cyclically quasi-monotone, and the lower contour sets of $\preceq_T$ are $[0,x]$ if $x \geq 0$ and $[x,0]$ if $x \leq 0$ (see \cite[Section~3]{depascaleRockafellarTheoremCyclically2025} for the one-dimensional case). This is why, in the present paper, we establish further properties of $\preceq_T$, notably when $T$ is maximal, that will allow us to then extend it into a preference relation.

\section{Maximal cyclically quasi-monotone operators and existence of potentials}\label{sec_existence}

In this section, we study the question of extending $\preceq_T$ into a preference relation whose lower contour sets are closed and convex and to which $T$ is normal (\Cref{thm_chain}). In the example of $T = \nop f$ with $f = \abs{\cdot}$ on $\R$ (see \Cref{sec_canonical_preorders}), unions of lower contour sets, such as $C^T(x) \cup C^T(-x) = [-\abs{x},\abs{x}]$, are the lower contour sets of a quasi-convex preference relation. This fails in general: in \cite[Example~5.5]{depascaleRockafellarTheoremCyclically2025}, some unions of lower contour sets of $\preceq_T$ are not convex. However, the operator of this example is not maximal. In \Cref{sec_maximal_operators}, we show that unions of lower contour sets of $\preceq_T$ are always convex when $T$ is maximal, which is why we first replace $T$ by a maximal extension in the proof of \Cref{thm_chain}. In \Cref{sec_extension_preorders}, we select such unions through a total extension of $\preceq_T$ that preserves its strict part, which keeps $T$ normal. We prove \Cref{thm_chain} in \Cref{sec_proof_existence}. When $\X$ is separable, representing the resulting preference relation by a function (\Cref{cor_rader_separable}) yields a quasi-convex potential, which proves item~\labelcref{main_thm_existence} of \Cref{main_thm}.

\subsection{Maximal operators and convexity of unions of lower contour sets}\label{sec_maximal_operators}

Recall that a cyclically quasi-monotone operator is maximal cyclically quasi-monotone if it admits no cyclically quasi-monotone extension other than itself.

\begin{prp}\label{existence_maximal_extension}
	Every cyclically quasi-monotone operator $T \colon \X \multimap \X^*$ admits a maximal cyclically quasi-monotone extension.
\end{prp}

\begin{proof}
	Consider the set
	\[\{S : S \text{ is cyclically quasi-monotone and } T \subseteq S\},\]
	ordered by inclusion, which is nonempty since it contains $T$. If $(T_i)_{i \in I}$ is a nonempty chain in this set, then $\bigcup_{i\in I} T_i$ is a cyclically quasi-monotone extension of $T$, since a cycle involves finitely many points, hence is contained in some $T_i$. Thus this set is inductive and, by Zorn's lemma, it admits a maximal element.
\end{proof}

\begin{lem}\label{F_equals_normal_cone}
	Let $T \colon \X \multimap \X^*$ be a cyclically quasi-monotone operator. Then
	\begin{equation}\label{F_included_normal_cone}
		\forall x \in \X, \quad T(x) \subseteq \nop_{C^T(x)}(x).
	\end{equation}
	If moreover $T$ is maximal cyclically quasi-monotone, then
	\[\forall x \in \X, \quad T(x) = \nop_{C^T(x)}(x) = (S^T(x) - x)^\circ.\]
\end{lem}

\begin{proof}
	The first part is \cite[Proposition~2.17~(ii)]{depascaleRockafellarTheoremCyclically2025}, whose proof we recall for completeness. Let $x \in \X$ and $x^* \in T(x)$. If we had $\pair{x^*}{y-x} > 0$ for some $y \in C^T(x)$, then $x \prec_T y \preceq_T x$, which implies $x \prec_T x$ by \Cref{transitivity_properties}: a contradiction. Thus $\pair{x^*}{y-x} \leq 0$ for every $y \in C^T(x)$, which means that $x^* \in \nop_{C^T(x)}(x)$, and \labelcref{F_included_normal_cone} is proved.

	Now assume that $T$ is maximal and let $x \in \X$. Since $S^T(x) \subseteq C^T(x)$, \labelcref{F_included_normal_cone} yields
	\[T(x) \subseteq \nop_{C^T(x)}(x) \subseteq (S^T(x) - x)^\circ,\]
	so that it suffices to show $(S^T(x) - x)^\circ \subseteq T(x)$. Let $x^* \in (S^T(x) - x)^\circ$ and define $\tilde T$ by $\tilde T(y) = T(y)$ for $y \neq x$ and $\tilde T(x) = T(x) \cup\{x^*\}$. Let us show that $\tilde T$ is still cyclically quasi-monotone.

	By contradiction, assume that there exists a $\tilde T$-ascending cycle $x_0,x_1, \ldots, x_n = x_0$ with corresponding $x_i^* \in \tilde T(x_i)$ for every $i < n$. Let us assume that it has minimal length, i.e., $n \geq 2$ is minimal, so that $x_i \neq x_j$ for every $0 \leq i \neq j < n$. Since $T$ is cyclically quasi-monotone, necessarily one of the $(x_i,x_i^*)$ is equal to $(x,x^*)$, say $(x_0,x_0^*) = (x,x^*)$. By minimality of $n$, for every $1 \leq i < n$, we have $x_i \neq x$ and thus $x_i^* \in \tilde T(x_i) = T(x_i)$. This implies that $x_1 \prec_T x_n = x$, i.e., $x_1 \in S^T(x)$. But $\pair{x^*}{x_1-x} = \pair{x_0^*}{x_1-x_0} > 0$, which contradicts $x^* \in (S^T(x) - x)^\circ$.

	Therefore $\tilde T$ is cyclically quasi-monotone, and since $T \subseteq \tilde T$ and $T$ is maximal, we have $T = \tilde T$, which implies that $x^* \in T(x)$. We get the result since $x^*$ was an arbitrary element of $(S^T(x) - x)^\circ$.
\end{proof}

The key property of maximal operators for our purposes is the following convexity property of the lower contour sets $C^T(x)$ of $\preceq_T$, which is the technical core of this section and relies on \Cref{F_equals_normal_cone}. In this lemma, as in \Cref{thm_chain} below, we work in smoothable Banach spaces. Recall that a Banach space $\X$ is \emph{smoothable} if it admits an equivalent norm that is Gâteaux-differentiable on $\X \setminus \{0\}$. This class is relevant for our purposes since it contains all reflexive spaces as well as all separable spaces, and we have seen that separability is necessary for the existence of potentials (\Cref{prp_nonseparable_counterexample}).

\begin{lem}\label{convexity_stable_by_union}
	Let $T \colon \X \multimap \X^*$ be a maximal cyclically quasi-monotone operator on a smoothable Banach space $\X$. Then for every set of points $A \subseteq \X$,
	\[\bigcup_{x\in A} C^T(x) \text{ is a convex set}.\]
\end{lem}
\begin{proof}
	Since convexity only involves pairs of points, it suffices to show that for every $x, x' \in A$, the set
	\[K \coloneqq C \cup C', \quad \text{where } C \coloneqq C^T(x) \text{ and } C' \coloneqq C^T(x'),\]
	is convex. Notice that $K$ is closed by \Cref{contour_sets_closed_convex}. Since an intersection of closed half-spaces is convex, it suffices to separate $K$ from any $y \notin K$ by a closed affine hyperplane, namely to find $v^* \in \X^*$ such that
	\begin{equation}\label{separation_union_convex_sets}
		\sup_{z \in K} \pair{v^*}{z} < \pair{v^*}{y}.
	\end{equation}
	Let us fix $y \notin K$, an equivalent norm $\norm{\cdot}$ on $\X$ that is Gâteaux-differentiable on $\X \setminus \{0\}$, and denote by $\norm{\cdot}^*$ its dual norm on $\X^*$ and by $\norm{\cdot}'$ its derivative.

	\paragraph{Simple case: when a projection exists.} Assume first that there exists a projection of $y$ onto $K$, i.e., a point $p \in K$ such that
	\[\norm{p-y} = \inf_{z \in K} \norm{z-y},\]
	which is always the case if $\X$ is finite-dimensional, since $K$ is closed. We set $u^* \coloneqq \norm{\cdot}'(p-y)$, which is well defined since $p \neq y$, and satisfies $\pair{u^*}{p-y} = \norm{p-y} > 0$. If $p \in C$, then $p$ minimizes $\norm{\cdot - y}$ on the convex set $C \subseteq K$, so that by Fermat's rule $-u^* \in \nop_C(p)$, and similarly with $C'$. We distinguish two cases, the case $p \in C' \setminus C$ being symmetric to the first one.

	\textit{Case $p \in C \setminus C'$.} We already know that $-u^* \in \nop_C(p)$, i.e., $\pair{-u^*}{z-p} \leq 0$ for every $z \in C$, and it remains to show that the same holds for $z \in C'$. First notice that, since $C^T(p) \subseteq C$ because $p \in C$, \Cref{F_equals_normal_cone} yields
	\[-u^* \in \nop_C(p) \subseteq \nop_{C^T(p)}(p) = T(p).\]
	Assume by contradiction that $\pair{-u^*}{c'-p} > 0$ for some $c' \in C'$. Then $p \prec_T c'$, and since $c' \preceq_T x'$, \Cref{transitivity_properties} implies $p \prec_T x'$, hence $p \preceq_T x'$, i.e., $p \in C'$: a contradiction with $p \notin C'$. Thus $\pair{-u^*}{z-p} \leq 0$ for every $z \in C'$, hence for every $z \in K$.

	\textit{Case $p \in C \cap C'$.} Then $-u^* \in \nop_C(p) \cap \nop_{C'}(p)$, i.e., $\pair{-u^*}{z-p} \leq 0$ for every $z \in C \cup C' = K$.

	In both cases, $v^* \coloneqq -u^*$ satisfies
	\[\sup_{z \in K} \pair{v^*}{z} \leq \pair{v^*}{p} < \pair{v^*}{y},\]
	which is \labelcref{separation_union_convex_sets}.

	\paragraph{General case: without projection.} In general, the existence of a projection is not guaranteed, thus we resort to Ekeland's variational principle instead. We set $r \coloneqq \inf_{z \in K} \norm{z-y}$, which is positive since $K$ is closed and $y \notin K$. By Ekeland's variational principle applied to the function $\varphi \colon z \mapsto \norm{z-y} + \chi_K(z)$, for every $\eps \in \mathopen]0,\eps_0\mathclose[$ where $\eps_0 \coloneqq 1/2$, there exists a point $p_\eps \in K$ such that $\norm{p_\eps-y} \leq r + \eps$ and
	\begin{equation}\label{ekeland_minimizer}
		p_\eps \in \argmin_{z \in \X} \norm{z-y} + \chi_K(z) + \eps \norm{z-p_\eps}.
	\end{equation}
	If $p_\eps \in C$, since $C \subseteq K$, we also have
	\[p_\eps \in \argmin_{z \in \X} \norm{z-y} + \chi_C(z) + \eps \norm{z-p_\eps}.\]
	This is a convex problem, thus by Fermat's rule and the sum rule (the qualification condition being satisfied because the norm is continuous), using the differentiability of the norm and the fact that $p_\eps \neq y$, we get
	\begin{gather}
		0 \in \partial \norm{\cdot}(p_\eps-y) + \partial \chi_C(p_\eps) + \eps\partial \norm{\cdot}(0),\notag\\
		\shortintertext{i.e.,}
		0 = u^*_\eps + v^*_\eps + \eps w^*_\eps \quad\text{where}\quad u^*_\eps = \norm{\cdot}'(p_\eps-y),\ v^*_\eps \in \nop_C(p_\eps),\ \norm{w^*_\eps}^* \leq 1.\notag
	\end{gather}
	Compared with the simple case, the direction $-u^*_\eps$ is perturbed by the term $\eps w^*_\eps$, but $v^*_\eps$ still separates $y$ from $p_\eps$. Indeed, since $\norm{u^*_\eps}^* = 1$ and $\pair{u^*_\eps}{p_\eps -y} = \norm{p_\eps -y}$, we have
	\begin{equation} \label{v_near_direction_y_minus_projection}
		\begin{aligned}
			\pair{v^*_\eps}{y-p_\eps} &= \pair{u^*_\eps}{p_\eps -y} + \eps \pair{w^*_\eps}{p_\eps -y}\\
			& \geq \norm{p_\eps-y} - \eps \norm{w^*_\eps}^* \norm{p_\eps-y}\\
			& \geq \frac{1}{2} \norm{p_\eps-y} \geq \frac{r}{2} > 0.
		\end{aligned}
	\end{equation}
	Similarly, if $p_\eps \in C'$, we get $0 = u^*_\eps + \tilde v^*_\eps + \eps \tilde w^*_\eps$ with $\tilde v^*_\eps \in \nop_{C'}(p_\eps)$ and $\norm{\tilde w^*_\eps}^* \leq 1$, and \labelcref{v_near_direction_y_minus_projection} holds with $\tilde v^*_\eps$ instead of $v^*_\eps$. We distinguish two cases.

	\textit{Case $p_\eps \notin C \cap C'$ for some $\eps \in \mathopen]0,\eps_0\mathclose[$.} By symmetry, we may assume that $p_\eps \in C \setminus C'$. The argument of the simple case applies with $p_\eps$ and $v^*_\eps$ instead of $p$ and $-u^*$. Indeed, $v^*_\eps \in \nop_C(p_\eps) \subseteq \nop_{C^T(p_\eps)}(p_\eps) = T(p_\eps)$, and if $\pair{v^*_\eps}{c'-p_\eps} > 0$ for some $c' \in C'$, then $p_\eps \prec_T c' \preceq_T x'$ yields $p_\eps \prec_T x'$, hence $p_\eps \in C'$: a contradiction. Hence, using \labelcref{v_near_direction_y_minus_projection},
	\[\sup_{z \in K} \pair{v^*_\eps}{z} \leq \pair{v^*_\eps}{p_\eps} < \pair{v^*_\eps}{y},\]
	and $v^* \coloneqq v^*_\eps$ satisfies \labelcref{separation_union_convex_sets}.

	\textit{Case $p_\eps \in C \cap C'$ for every $\eps \in \mathopen]0,\eps_0\mathclose[$.} In that case, the first-order optimality conditions on $C$ and on $C'$ both hold, i.e.,
	\[0 = u^*_\eps + v^*_\eps + \eps w^*_\eps = u^*_\eps+\tilde v^*_\eps+\eps \tilde w^*_\eps.\]
	Contrary to the simple case, we cannot conclude directly, since $v^*_\eps$ and $\tilde v^*_\eps$ may differ. This is where the differentiability of the norm plays a crucial role: since $u^*_\eps$ is the same in both conditions, we get
	\begin{equation}\label{directions_alignment}
		\norm{v^*_\eps-\tilde v^*_\eps}^* = \eps\norm{w^*_\eps-\tilde w^*_\eps}^* \leq 2\eps.
	\end{equation}
	Since $\norm{v^*_\eps}^* \leq \norm{u^*_\eps}^* + \eps \leq 2$, the family $(v^*_\eps)_{\eps \in \mathopen]0,\eps_0\mathclose[}$ is bounded and by the Banach--Alaoglu theorem, there exists a weak-$*$ limit $v^*$ along some subnet $(v^*_{\eps_i})_{i \in I}$ such that $\eps_i \to 0$ along $I$. By \labelcref{directions_alignment} we have at the same time $v^*_{\eps_i} \rightharpoonup v^*$ and $\tilde v^*_{\eps_i} \rightharpoonup v^*$ weakly-$*$ along $I$. Now let us show that $v^*$ satisfies \labelcref{separation_union_convex_sets}. Take $c \in K$, for example $c \in C$. By \labelcref{v_near_direction_y_minus_projection} and since $\pair{v^*_\eps}{c-p_\eps} \leq 0$, we have
	\[\pair{v^*_\eps}{y-c} \geq \pair{v^*_\eps}{y-p_\eps} \geq \frac{r}{2}.\]
	Taking $\eps = \eps_i$ and passing to the limit along $i \in I$ yields $\pair{v^*}{y-c} \geq r/2$, and the same holds for $c \in C'$, using $\tilde v^*_{\eps_i}$ instead of $v^*_{\eps_i}$. We have shown that
	\[\sup_{z\in K} \pair{v^*}{z} \leq \pair{v^*}{y} - \frac{r}{2} < \pair{v^*}{y}. \qedhere\]
\end{proof}

\subsection{Extension of preorders and normality}\label{sec_extension_preorders}

To select unions of lower contour sets, we extend $\preceq_T$ into a total preorder by means of the following classical result, which generalizes to preorders Szpilrajn's theorem \cite{szpilrajnLextensionLordrePartiel1930}, stating that every partial order can be extended into a total order.

\begin{thm}[Arrow--Hansson {\cite{arrowSocialChoiceIndividual1951,hanssonChoiceStructuresPreference1968}}]\label{pre-order_extension}
	For every preorder $\preceq$ on a set $X$, there exists a total preorder $\preceq'$ on $X$ that extends $\preceq$ and preserves its strict part, i.e., such that for every $x,y \in X$,
	\begin{align*}
		x \preceq y &\implies x \preceq' y,\\
		(x \preceq y \text{ and } y \not\preceq x) &\implies y \not\preceq' x.
	\end{align*}
\end{thm}

Preserving the strict part is essential: without this requirement, the trivial relation for which all points are equivalent would be a total preorder extending $\preceq$. We refer to \cite{bossertConsistencyChoiceRationality2010,catoSzpilrajnArrowSuzumura2012} for modern proofs of this result and of related extension theorems. For the record, $\preceq'$ can moreover be chosen so as to create no new equivalent points, meaning that two points equivalent for $\preceq'$ are already equivalent for $\preceq$: it suffices to apply Szpilrajn's theorem to the partial order induced by $\preceq$ on the quotient of $X$ by its equivalence relation, and to lift the resulting total order to $X$.

In the next lemma, we show that $T$ is normal to the lower contour sets of a preference relation $\preceq$ if and only if $\prec_T$ is contained in the strict part of $\preceq$.

\begin{lem}\label{normality_strict_part}
	Let $T \colon \X \multimap \X^*$ be an operator and $\preceq$ be a preference relation on $\X$. The following assertions are equivalent.
	\begin{enumerate}[(i)]
		\item\label{item_normality} $T \subseteq \nop_\preceq$.
		\item\label{item_strict_part} $\prec_T$ is contained in the strict part of $\preceq$, i.e., for every $x, y \in \X$,
		\[x \prec_T y \implies y \not\preceq x.\]
	\end{enumerate}
\end{lem}
\begin{proof}
	Assume \labelcref{item_strict_part}, and let $x \in \X$, $x^* \in T(x)$ and $y \preceq x$. If we had $\pair{x^*}{y-x} > 0$, then $x \prec_T y$, hence $y \not\preceq x$ by \labelcref{item_strict_part}: a contradiction. Thus $x^* \in \nop_{C_\preceq(x)}(x) = \nop_\preceq(x)$ by \labelcref{normal_cone_operator_preference}, and \labelcref{item_normality} holds.

	Conversely, assume \labelcref{item_normality}, and let $x, y \in \X$ be such that $x \prec_T y$. Consider a $T$-ascending path $(x_0, \ldots, x_N)$ from $x$ to $y$, with $T$-ascending directions $(x_0^*, \ldots, x_{N-1}^*)$. For every $i \in \{0, \ldots, N-1\}$, since $x_i^* \in \nop_\preceq(x_i) = \nop_{C_\preceq(x_i)}(x_i)$ by \labelcref{item_normality} and $\pair{x_i^*}{x_{i+1}-x_i} > 0$, we have $x_{i+1} \not\preceq x_i$, hence $x_i \preceq x_{i+1}$ by totality. If we had $y \preceq x$, then $x_1 \preceq \cdots \preceq x_N = y \preceq x = x_0$ would give $x_1 \preceq x_0$ by transitivity: a contradiction. Thus $y \not\preceq x$, and \labelcref{item_strict_part} holds.
\end{proof}

\subsection{Existence of a quasi-convex preference relation}\label{sec_proof_existence}

\begin{thm}[Existence of a preference relation]\label{thm_chain}
	Let $T \colon \X \multimap \X^*$ be a cyclically quasi-monotone operator on a smoothable Banach space $\X$. Then there exists a lower semicontinuous quasi-convex preference relation $\preceq$ on $\X$ such that $T \subseteq \nop_\preceq$.
\end{thm}
\begin{proof}[Proof of \Cref{thm_chain}]
	First of all, up to replacing $T$ by a maximal extension, which is possible by \Cref{existence_maximal_extension}, we can assume $T$ to be maximal cyclically quasi-monotone. We then consider a total preorder $\preceq'$ that extends $\preceq_T$ and preserves its strict part, as given by \Cref{pre-order_extension}.

	Let us show that $\preceq'$ is quasi-convex. We claim that for every $x \in \X$,
	\begin{equation}\label{C_t_is_union_C}
		C_{\preceq'}(x) = \bigcup_{x' : x' \preceq' x} C^T(x').
	\end{equation}
	The inclusion $C_{\preceq'}(x) \subseteq \bigcup_{x' : x' \preceq' x} C^T(x')$ is trivial because $x' \in C^T(x')$ for every $x'$. For the converse inclusion, consider $x' \preceq' x$ and $z \in C^T(x')$, meaning $z \preceq_T x'$. Since $\preceq'$ extends $\preceq_T$, we deduce that $z\preceq' x'$, and by transitivity of $\preceq'$ we get $z \preceq' x$. We have proved \labelcref{C_t_is_union_C}, which implies by \Cref{convexity_stable_by_union} that $C_{\preceq'}(x)$ is convex, for every $x$.

	Now, since $\prec_T$ is contained in the strict part of $\preceq_T$ (\Cref{strict_part}), which is preserved by $\preceq'$, \Cref{normality_strict_part} yields $T \subseteq \nop_{\preceq'}$.

	Finally, by \Cref{lsc_regularization}, the relation $\preceq \coloneqq \overline{\preceq'}$ is a lower semicontinuous quasi-convex preference relation, and $T \subseteq \nop_{\preceq'} \subseteq \nop_\preceq$.
\end{proof}

As a consequence of \Cref{thm_chain}, the points where a cyclically quasi-monotone operator contains a large vector subspace lie in a small affine subspace. For an operator $T \colon \X \multimap \X^*$ and an integer $k \geq 1$, we set
\[
	B_k(T) \coloneqq \{x \in \X : T(x) \text{ contains a vector subspace of dimension } k\},
\]
so that $B_{k+1}(T) \subseteq B_k(T)$ for every $k \geq 1$.

\begin{cor}[Stratification]\label{stratification}
	Let $T \colon \X \multimap \X^*$ be a cyclically quasi-monotone operator on a smoothable Banach space $\X$. For every integer $k \geq 1$, the set $B_k(T)$ is contained in a closed affine subspace of $\X$ of codimension at least $k$. In particular, if $\X$ has finite dimension $d$, then $B_k(T)$ is contained in an affine subspace of dimension at most $d-k$.
\end{cor}

\begin{proof}
	By \Cref{thm_chain}, there exists a preference relation $\preceq$ on $\X$ such that $T \subseteq \nop_\preceq$, i.e., $T(x) \subseteq \nop_{C_\preceq(x)}(x)$ for every $x \in \X$ by \labelcref{normal_cone_operator_preference}. Let $x \in B_k(T)$ and let $L_x \subseteq T(x)$ be a vector subspace of dimension $k$. For every $x^* \in L_x$ and every $y \in C_\preceq(x)$, we have $\pair{x^*}{y-x} = 0$ since $\pm x^* \in \nop_{C_\preceq(x)}(x)$, thus $C_\preceq(x) \subseteq x + L_x^\perp$, which is a closed affine subspace of codimension $k$. As a consequence, the closed affine hull $E_x$ of $C_\preceq(x)$ has codimension at least $k$.

	Since the lower contour sets of $\preceq$ form a chain, so does the family $(E_x)_{x \in B_k(T)}$. If $B_k(T)$ is empty, there is nothing to prove. Otherwise, pick $x_0 \in B_k(T)$ such that the codimension $c$ of $E_{x_0}$ is minimal. For every $x \in B_k(T)$, either $E_x \subseteq E_{x_0}$, or $E_{x_0} \subseteq E_x$, in which case the codimension of $E_x$ is at most $c$, hence equal to $c$ by minimality, so that $E_x = E_{x_0}$ as nested closed affine subspaces of the same finite codimension. In both cases $x \in E_x \subseteq E_{x_0}$, and we conclude that $B_k(T) \subseteq E_{x_0}$, a closed affine subspace of codimension $c \geq k$.
\end{proof}

\section{Characterization of maximal cyclically quasi-monotone operators}\label{sec_maximal}

In this section, we assume that $\X$ has finite dimension. For a subset $A$ of $\X$ (resp.\ of $\X^*$), we denote by $A^\perp$ its orthogonal, i.e., the set of elements of $\X^*$ (resp.\ of $\X$) that vanish on $A$.

The \emph{lineality space} of a convex cone $\Lambda$ \cite[Section~8]{rockafellarConvexAnalysis1970} is
\[\lin \Lambda \coloneqq \Lambda \cap (-\Lambda),\]
which is the largest vector subspace included in $\Lambda$. In particular, $\Lambda$ is a vector subspace if and only if $\Lambda = \lin \Lambda$.

We first show a nondegeneracy property of maximal cyclically quasi-monotone operators: their values are never proper vector subspaces (\Cref{sec_nondegeneracy}). For normal cone operators, this property characterizes a natural class of quasi-convex functions, which we call essentially non-satiated (\Cref{sec_ens}). Conversely, the normal cone operator of such a function is maximal, even among quasi-monotone operators (\Cref{sec_maximality_normal_cone}). We conclude with the proof of \Cref{main_thm} (\Cref{sec_proof_thm_A}).

\subsection{Nondegeneracy of maximal operators}\label{sec_nondegeneracy}

Let $T$ be a cyclically quasi-monotone operator, and recall that $S^T(x) = \{\cdot \prec_T x\}$ denotes the set of points that are \enquote{strictly smaller} than $x$ according to $T$. If $T(\bar x) = \X^*$ for some $\bar x \in \X$, then $\bar x \prec_T y$ for every $y \neq \bar x$. This happens at most at one point, and we now focus on the points $x$ such that $T(x) \neq \X^*$.

\begin{prp}\label{strict_sublevel_half_subspace}
	Let $T$ be a cyclically quasi-monotone operator whose values are convex cones, and let $x \in \X$ be such that $T(x) \neq \X^*$. Setting $E \coloneqq \lin T(x)$, the set $S^T(x)$ is included in a closed half-space of $x + E^\perp$ whose boundary contains $x$, i.e., there exists $x^* \in \X^* \setminus E$ such that
	\[
		S^T(x) \subseteq \{ y : \pair{x^*}{y - x} \leq 0\} \cap (x+E^\perp).
	\]
\end{prp}

\begin{proof}
	Since $T(x) \neq \X^*$, $E = \lin T(x)$ is a proper vector subspace of $\X^*$. First of all, let us show that $C^T(x)$, and thus $S^T(x)$ as well, is included in $x +E^\perp$. We proceed by contraposition, considering $y \not\in x+E^\perp$. There exists $e^* \in E$ such that $\pair{e^*}{y-x} \neq 0$, and since $\pm e^* \in T(x)$ we get $\pair{y^*}{y-x} > 0$ for either $y^* = e^*$ or $y^* = -e^*$, thus $x \prec_T y$. By \Cref{transitivity_properties}, this implies that $y\not\in C^T(x)$.

	To show the desired result, it suffices to prove that the convex cone $\cc(S^T(x)-x)$, which is included in $E^\perp$, is not equal to $E^\perp$. Indeed, in that case, considering the associated polar cones in $E^\perp$, there exists a linear form on $E^\perp$ which is nonpositive on $S^T(x)-x$ and does not vanish identically, and any extension $x^* \in \X^*$ of this form satisfies the desired inclusion, as well as $x^* \notin E$. We proceed by contradiction, assuming that $\cc(S^T(x)-x) = E^\perp$. Thus there exists a finite subset $\mathcal{P}\subseteq S^T(x)$ such that $\cc(\mathcal{P}-x) = E^\perp$.

	Take $x_0 \in \mathcal{P}$. Since $x_0 \prec_T x$, there exists a $T$-ascending path $(z_0, \ldots, z_n)$ from $x_0$ to $x$. In particular, there exists $z_{n-1}^* \in T(z_{n-1})$ such that
	\begin{equation}\label{x_in_half_space}
		\pair{z_{n-1}^*}{x-z_{n-1}} > 0.
	\end{equation}
	Besides, if $n \geq 2$, then $x_0 \prec_T z_{n-1}$, and since $T$ is cyclically quasi-monotone, necessarily
	\begin{equation}\label{x0_in_half_space}
		\pair{z_{n-1}^*}{x_0 - z_{n-1}} \leq 0,
	\end{equation}
	which also holds trivially if $n = 1$, since then $z_{n-1} = x_0$.

	We claim that
	\begin{equation}\label{induction_step}
		\exists x_1 \in \mathcal{P}\setminus\{x_0\} \quad\text{such that}\quad x_0 \prec_T x_1.
	\end{equation}
	By contradiction, assume this is false. In that case we necessarily have
	\begin{equation}\label{x0_complement_in_half_space}
		\mathcal{P}\setminus\{x_0\} \subseteq \{y : \pair{z_{n-1}^*}{y-z_{n-1}} \leq 0\},
	\end{equation}
	because otherwise there would exist a point $y \in \mathcal{P}\setminus\{x_0\}$ such that $\pair{z_{n-1}^*}{y-z_{n-1}} > 0$, thus $z_{n-1} \prec_T y$, hence $x_0 \prec_T y$. From \labelcref{x0_in_half_space} and \labelcref{x0_complement_in_half_space} we get
	\[\mathcal{P} \subseteq \{y : \pair{z_{n-1}^*}{y-z_{n-1}} \leq 0\},\]
	and since $\{y : \pair{z_{n-1}^*}{y-z_{n-1}} \leq 0\} \subseteq \{y : \pair{z_{n-1}^*}{y-x} \leq 0\}$ by \labelcref{x_in_half_space}, we deduce that
	\[\mathcal{P} - x \subseteq \{v \in E^\perp : \pair{z_{n-1}^*}{v} \leq 0\}.\]
	Since $z_{n-1} \in S^T(x) \subseteq x + E^\perp$, we have $x - z_{n-1} \in E^\perp$, so that the linear form $\pair{z_{n-1}^*}{\cdot}$ does not vanish identically on $E^\perp$ by \labelcref{x_in_half_space}. Thus the right-hand side is a closed half-space of $E^\perp$, which contradicts the fact that $\cc(\mathcal{P}-x) = E^\perp$, and we have proved \labelcref{induction_step}.

	Since $x_0 \in \mathcal{P}$ was arbitrary, every point of $\mathcal{P}$ is $\prec_T$-smaller than another point of $\mathcal{P}$. Starting from any point of $\mathcal{P}$, we may thus build a sequence $(x_i)_{i \in \N}$ in $\mathcal{P}$ such that $x_i \prec_T x_{i+1}$ for every $i$. Since $\mathcal{P}$ is finite, we have $x_j = x_i$ for some $i < j$, hence $x_i \prec_T x_{i+1} \prec_T \cdots \prec_T x_j = x_i$ yields a $T$-ascending cycle: a contradiction with $T$ being cyclically quasi-monotone.
\end{proof}

\begin{cor}\label{values_maximal_operators}
	If $T$ is maximal cyclically quasi-monotone, then for every $x \in \X$, $T(x)$ is not a proper vector subspace of $\X^*$.
\end{cor}
\begin{proof}
	Recall that the values of $T$ are closed convex cones by \Cref{F_equals_normal_cone}. Assume that $T(x) \neq \X^*$. By \Cref{strict_sublevel_half_subspace}, there exists $x^* \in \X^* \setminus \lin T(x)$ such that $\pair{x^*}{y - x} \leq 0$ for every $y \in S^T(x)$, i.e., $x^* \in (S^T(x)-x)^\circ$. Since $T$ is maximal, $(S^T(x)-x)^\circ = T(x)$ by \Cref{F_equals_normal_cone}, hence $x^* \in T(x) \setminus \lin T(x)$, and $T(x)$ is not a vector subspace.
\end{proof}

\subsection{Essentially non-satiated functions}\label{sec_ens}

By \Cref{values_maximal_operators}, if $\nop f$ is maximal cyclically quasi-monotone for some quasi-convex function $f$, then $\nop f(x)$ is never a proper vector subspace of $\X^*$. Since $\nop f(x) = \X^*$ if and only if $x$ is a (hence the unique) strict global minimizer of $f$, this means that $\nop f(x)$ is not a vector subspace at any other point $x$. We will show that this characterizes the following class of quasi-convex functions.

\begin{dfn}[Essential non-satiation]\label{dfn_essential_non_satiation}
	A quasi-convex function $f \colon \X \to \R \cup \{+\infty\}$ is said to be \emph{essentially non-satiated} if every local minimizer of $f$ in $\dom f$ is a strict global minimizer. In particular, if $f$ has a minimizer in $\dom f$, then it is unique.
\end{dfn}

\begin{rmk}[Terminology and related notions]\label{rmk_non_satiation}
	The terminology comes from economics, where a preference relation is \emph{locally non-satiated} if every neighborhood of every point contains a strictly preferred point \cite[Definition~3.B.3]{mascolellMicroeconomicTheory1995}. For a utility function, this means that it has no local maximizer, and for $f$, that it has no local minimizer. Essential non-satiation allows one exception, namely a unique global minimizer, which corresponds to a \emph{bliss point} in economics. In optimization, quasi-convex functions whose local minimizers in $\dom f$ are all global minimizers are called \emph{neatly quasi-convex} \cite[Proposition~4.1]{alhomidanTransformationQuasiconvexFunctions2018}. In \cite[Definition~4.1]{alhomidanTransformationQuasiconvexFunctions2018}, they are defined equivalently by the fact that $\{f < f(x)\}$ and $\{f \leq f(x)\}$ have the same closure, or the same relative interior, for every $x$ such that $f(x) > \inf f$. Thus $f$ is essentially non-satiated if and only if it is neatly quasi-convex and has at most one minimizer. Neat quasi-convexity is weaker than the classical \emph{semistrict quasi-convexity} \cite{karamardianStrictlyQuasiconvexConcave1967,avrielGeneralizedConcavity1988}, which requires $f(z) < f(y)$ for every $x, y \in \dom f$ such that $f(x) < f(y)$ and every $z \in \mathopen]x,y\mathclose[$. Both notions coincide for continuous functions \cite[Theorem~3.37]{avrielGeneralizedConcavity1988}, \cite[Proposition~4.2]{alhomidanTransformationQuasiconvexFunctions2018}, but not for lower semicontinuous ones \cite[Example~4.1]{alhomidanTransformationQuasiconvexFunctions2018}.
\end{rmk}

We denote by $\relint C$ the relative interior of a convex set $C \subseteq \X$, i.e., its interior relative to its affine hull. We shall use the following classical properties.

\begin{prp}[{\cite[Theorems~6.2, 6.3 and 6.4]{rockafellarConvexAnalysis1970}}]\label{relative_interior}
	Let $C$ be a nonempty convex subset of $\X$. The following statements hold.
	\begin{enumerate}[(a)]
		\item\label{item_ri_dense} $\relint C$ is nonempty and dense in $\cl C$.
		\item\label{item_ri_cone} For every $x \in C$, we have $x \in \relint C$ if and only if $\cc (C-x)$ is a vector subspace of $\X$, if and only if $\clcc (C-x)$ is a vector subspace of $\X$.
	\end{enumerate}
\end{prp}

\begin{lem}\label{equivalence_normal_cone_non_linear}
	Let $f \colon \X \to \R \cup \{+\infty\}$ be a quasi-convex function and $x \in \dom f$. Then $\nop f(x)$ is a vector subspace of $\X^*$ if and only if $x \in \relint \{f \leq f(x)\}$.
\end{lem}
\begin{proof}
	Setting $K \coloneqq \{f \leq f(x)\}-x$ and $C \coloneqq \clcc K$, we have $\nop f(x) = K^\circ = C^\circ$. Since $C$ is a closed convex cone, we have $C = C^{\circ\circ}$, and since the polar cone of a vector subspace is its orthogonal, $C$ is a vector subspace if and only if $C^\circ$ is. We conclude by \Cref{relative_interior} \labelcref{item_ri_cone}.
\end{proof}

\begin{prp}\label{characterization_essential_non-satiation}
	A quasi-convex function $f$ is essentially non-satiated if and only if for every $x \in \dom f$, $\nop f(x)$ is not a proper vector subspace of $\X^*$.
\end{prp}
\begin{proof}
	Assume that $f$ is essentially non-satiated. In particular, $f$ admits at most one minimizer, thus
	\[x \text{ is a global minimizer } \iff x \text{ is a strict global minimizer } \iff \nop f(x) = \X^*.\]
	It suffices to show that for any point $x$ which is not a global minimizer, $\nop f(x)$ is not a vector subspace of $\X^*$. By \Cref{equivalence_normal_cone_non_linear}, this is equivalent to $x \notin \relint \{f \leq f(x)\}$. Assume by contradiction that for some $x$ which is not a minimizer of $f$, we have $x \in \relint \{f \leq f(x)\}$. This does not imply that $x$ is a local minimizer of $f$, but we show that there exists a point $x'$ such that $f(x') = f(x)$ and $x'$ is a local minimizer of $f$. Since in finite dimension we may always separate two disjoint convex sets by a linear form (in the weak sense), we may separate $x$ and $\{f < f(x)\}$ linearly in the affine span $A$ of $\{f \leq f(x)\}$, and extending the corresponding linear form to an element of $\X^*$ we obtain a linear form $x^* \in \X^*$ such that
	\[\sup_{y\in \{f < f(x)\}} \pair{x^*}{y-x} \leq 0\]
	and $x^*$ is not identically $0$ on $\{f \leq f(x)\}-x$. Since $x \in \relint \{f \leq f(x)\}$, we see that $H_+ \coloneqq \{y : \pair{x^*}{y-x} > 0\}$ has a nonempty intersection with $\{f \leq f(x)\}$ (just take $y = x+tu$ for $t > 0$ small and $\pair{x^*}{u} > 0$ for some well-chosen $u$ in the linear span of $\{f \leq f(x)\}-x$). By definition of $x^*$, we have $f \geq f(x)$ on the open set $H_+$. Consequently, any point $x' \in H_+ \cap \{f \leq f(x)\}$ satisfies $f(x') = f(x)$, and it is a local minimizer of $f$, since $f \geq f(x')$ on its open neighborhood $H_+$. Since $f(x') = f(x)$ and $x$ was assumed not to be a global minimizer, then $x'$ is not a global minimizer either, contradicting the fact that $f$ is essentially non-satiated.

	Conversely, assume that $\nop f(x)$ is never a proper vector subspace of $\X^*$, and let $x \in \dom f$ be a local minimizer of $f$. Setting $C \coloneqq \{f \leq f(x)\}$, there exists $r > 0$ such that $f = f(x)$ on $C \cap B(x,r)$. By \Cref{relative_interior} \labelcref{item_ri_dense}, there exists $y \in \relint C \cap B(x,r)$. Then $f(y) = f(x)$, thus $\{f \leq f(y)\} = C$ and $y \in \relint \{f \leq f(y)\}$, so that $\nop f(y)$ is a vector subspace by \Cref{equivalence_normal_cone_non_linear}. By assumption, $\nop f(y) = \X^*$, i.e., $y$ is a strict global minimizer of $f$, so that $x = y$ is also a strict global minimizer of $f$, and $f$ is essentially non-satiated.
\end{proof}

\subsection{Maximality of normal cone operators}\label{sec_maximality_normal_cone}

Given an operator $T$ and $x \in \X$, we denote by
\[S^T_1(x) \coloneqq \{y \in \X : \pair{y^*}{x-y} > 0 \text{ for some } y^* \in T(y)\} \subseteq S^T(x)\]
the set of points from which $x$ can be reached by a $T$-ascending path of length one.

\begin{lem}\label{lem_equality_cones}
	Let $f$ be a lower semicontinuous quasi-convex function which is essentially non-satiated. If $x \notin \argmin f$ then
	\[\clcc (S^{\nop f}_1(x)-x) = \clcc (S^{\nop f}(x)-x) = \clcc (\{f < f(x)\}-x) = \clcc (\{f \leq f(x)\}-x).\]
\end{lem}
\begin{proof}
	Notice that $S^{\nop f}_1(x) \subseteq S^{\nop f}(x) \subseteq \{f < f(x)\} \subseteq \{f \leq f(x)\}$. Besides, since $\{f \leq f(x)\}$ is nonempty, closed and convex, \Cref{relative_interior} \labelcref{item_ri_dense} yields
	\[\cl \relint \{f \leq f(x)\} = \{f \leq f(x)\}.\]
	To get the equalities of all considered cones, it thus suffices to show that
	\[\relint \{f \leq f(x)\}-x \subseteq \clcc (S^{\nop f}_1(x)-x).\]

	The essential non-satiation of $f$ guarantees that
	\[\relint \{f \leq f(x)\} \subseteq \{f < f(x)\}.\]
	Indeed, by contraposition, if $y$ is a point such that $f(y) = f(x)$, then by \Cref{characterization_essential_non-satiation} and the fact that $x$ is not a minimizer of $f$ (hence neither is $y$), $\nop f(y)$ is not a vector subspace of $\X^*$, and by \Cref{equivalence_normal_cone_non_linear} it implies that $y \notin \relint \{f \leq f(x)\}$. In other words, $\{f < f(x)\}$ and $\{f \leq f(x)\}$ have the same relative interior (see \cite[Proposition~4.1]{alhomidanTransformationQuasiconvexFunctions2018} and \Cref{rmk_non_satiation}).

	Let us show that for every $y \in \relint \{f \leq f(x)\}$, for $t > 0$ small enough we have $x+t(y-x) \in S^{\nop f}_1(x)$. By contradiction, assume that $x_n \coloneqq x+t_n(y-x) \notin S^{\nop f}_1(x)$ for some sequence $t_n \downarrow 0$. Since $f$ is lower semicontinuous and $x_n \to x$, we get $f(x) \leq \liminf_n f(x_n)$. But $x_n$ lies on the segment between $x \in \{f \leq f(x)\}$ and $y \in \relint \{f \leq f(x)\}$, so that $x_n \in \relint \{f \leq f(x)\} \subseteq \{f < f(x)\}$ for every $n$, and thus $f(x_n) \to f(x)$. Therefore
	\begin{equation}\label{union_strict_sublevels}
		\{f < f(x)\} = \bigcup_{n\in \N} \{f \leq f(x_n)\}.
	\end{equation}
	By \labelcref{union_strict_sublevels} and the fact that $\{f < f(x)\}$ is dense in $\{f \leq f(x)\}$, the affine span of $\{f \leq f(x_n)\}$ coincides with that of $\{f \leq f(x)\}$ for $n$ large enough, which we assume from now on. Denoting by $V$ the direction of this affine span, we get $\lin \nop f(x_n) = (\{f \leq f(x_n)\}-x_n)^\perp = V^\perp$.

	For every $n$, since $x_n$ is not a minimizer of $f$, the cone $\nop f(x_n)$ is not a vector subspace by \Cref{characterization_essential_non-satiation}, hence it contains some $y_n^* \notin V^\perp$. Let $p_n$ be a closest point to $y_n^*$ in $V^\perp$. Since $-p_n \in V^\perp = \lin \nop f(x_n)$, the linear form $y_n^* - p_n$ still belongs to $\nop f(x_n)$, and $\operatorname{dist}(y_n^* - p_n, V^\perp) = \operatorname{dist}(y_n^*, V^\perp) = \norm{y_n^* - p_n} > 0$. Dividing by this quantity, we obtain $x_n^* \in \nop f(x_n)$ such that $\norm{x_n^*} = \operatorname{dist}(x_n^*, V^\perp) = 1$. From $x_n \notin S^{\nop f}_1(x)$ we obtain $\pair{x_n^*}{x-x_n} \leq 0$, or equivalently
	\[\pair{x_n^*}{y-x} \geq 0.\]
	Besides, since $f(y) < f(x) = \lim_n f(x_n)$ we know that for $n$ large enough $y \in \{f \leq f(x_n)\}$, thus $\pair{x_n^*}{y-x_n} \leq 0$, or equivalently
	\[ \pair{x_n^*}{y-x} \leq 0.\]
	As a consequence $\pair{x_n^*}{y-x} = 0$ for every $n$ large enough. By relative compactness of bounded subsets in finite dimension, we can extract a subsequence (relabeled) $(x_n^*)$ which converges to some $x^*$ such that $\operatorname{dist}(x^*, V^\perp) = 1$, and which still satisfies
	\[\pair{x^*}{y-x} = 0.\]
	Besides, $x_n^*$ being in $\nop f(x_n)$ we have $\pair{x_n^*}{z-x_n} \leq 0$ for every $z$ in $\{f \leq f(x_n)\}$ thus for every $z \in \{f < f(x)\}$ for $n$ large enough, and by passing to the limit we get
	\[\forall z \in \{f < f(x)\}, \quad \pair{x^*}{z-x} \leq 0.\]
	By density of $\relint \{f \leq f(x)\}$, thus of $\{f < f(x)\}$, in $\{f \leq f(x)\}$, it means that $x^* \in \nop f(x)$. We therefore get a contradiction: since $\pair{x^*}{y-x} = 0$, $x^*$ attains its maximum on $\{f \leq f(x)\}$ at $x$ but also at $y$, and $y$ is a relative interior point of $\{f \leq f(x)\}$, thus $x^*$ is constant on that set, which means that $x^* \in V^\perp$, contradicting $\operatorname{dist}(x^*, V^\perp) = 1$.
\end{proof}

\begin{thm}[Characterization of maximal operators]\label{thm_characterization_maximal}
	Let $f \colon \X \to \R \cup \{+\infty\}$ be a proper lower semicontinuous quasi-convex function. The following assertions are equivalent.
	\begin{enumerate}[(i)]
		\item $\nop f$ is maximal quasi-monotone.
		\item $\nop f$ is maximal cyclically quasi-monotone.
		\item $f$ is real-valued and essentially non-satiated.
	\end{enumerate}

	Moreover, if $T$ is a cyclically quasi-monotone operator, the following assertions are equivalent.
	\begin{enumerate}[(i')]
		\item $T$ is maximal quasi-monotone.
		\item $T$ is maximal cyclically quasi-monotone.
		\item $T = \nop f$ for some $f \in \cE_0(\X)$.
	\end{enumerate}
\end{thm}
\begin{proof}
	If $\nop f$ is maximal quasi-monotone, then it is maximal cyclically quasi-monotone, since it is cyclically quasi-monotone (\Cref{convex_chain_operator_CQM}) and cyclically quasi-monotone operators are quasi-monotone. If $\nop f$ is maximal cyclically quasi-monotone, then $0 \in \nop f(x)$ for every $x \in \X$, since adding $(x,0)$ to $\nop f$ preserves cyclical quasi-monotonicity. Thus $f$ is real-valued, since $\nop f(x) = \emptyset$ whenever $f(x) = +\infty$. Moreover, $\nop f(x)$ is never a proper vector subspace of $\X^*$ by \Cref{values_maximal_operators}, which means that $f$ is essentially non-satiated by \Cref{characterization_essential_non-satiation}.

	Now assume that $f$ is real-valued and essentially non-satiated, and let $x \in \X$ and $x^* \in \X^*$ be such that $\nop f \cup \{(x,x^*)\}$ is quasi-monotone. Then $\pair{x^*}{y-x} \leq 0$ for every $y \in S^{\nop f}_1(x)$, i.e., $x^* \in (S^{\nop f}_1(x)-x)^\circ$. If $x$ is the minimizer of $f$, then $\nop f(x) = \X^*$ contains $x^*$. Otherwise, \Cref{lem_equality_cones} yields
	\[(S^{\nop f}_1(x)-x)^\circ = (\{f \leq f(x)\}-x)^\circ = \nop f(x),\]
	thus $x^* \in \nop f(x)$. We have shown that $\nop f$ is maximal quasi-monotone.
	
	Now let $T$ be a cyclically quasi-monotone operator. If $T$ is maximal quasi-monotone, then it is maximal cyclically quasi-monotone, since cyclically quasi-monotone operators are quasi-monotone. If $T$ is maximal cyclically quasi-monotone, then $T \subseteq \nop f$ for some $f \in \cQ_0(\X)$ by item~\labelcref{main_thm_existence} of \Cref{main_thm}, and $T = \nop f$ by maximality, since $\nop f$ is cyclically quasi-monotone (\Cref{convex_chain_operator_CQM}). Then $f \in \cE_0(\X)$ by the implication (ii) $\Rightarrow$ (iii). Finally, if $T = \nop f$ for some $f \in \cE_0(\X)$, then $T$ is maximal quasi-monotone by the implication (iii) $\Rightarrow$ (i), which ends the proof.
\end{proof}
The fact that the equality of the cones of \Cref{lem_equality_cones}, which we proved for essentially non-satiated functions, implies maximal quasi-monotonicity was already observed for general operators in \cite{buenoMaximalityQuasimonotoneOperators2019}.

\subsection{Proof of \texorpdfstring{\Cref{main_thm}}{Theorem A}}\label{sec_proof_thm_A}

We are now almost in position to prove \Cref{main_thm}, by assembling the results of the previous sections. It only remains to establish the nonuniqueness statement, which is the object of the following example.

\begin{xmp}[Nonuniqueness]\label{example_non_uniqueness}
	Consider on $\R$ the functions $f_1(x) \coloneqq \abs{x}$ and
	\[
		f_2(x) \coloneqq
		\begin{cases}
			-x & \text{if } x \leq 0,\\
			2x & \text{if } x \geq 0.
		\end{cases}
	\]
	Both belong to $\cE_0(\R)$: they are continuous and quasi-convex, and $0$ is their only local minimizer, which is a strict global minimizer. For $t \geq 0$, their sublevel sets are $\{f_1 \leq t\} = [-t, t]$ and $\{f_2 \leq t\} = [-t, t/2]$. Hence, for every $x \in \R$,
	\[
		\nop f_1(x) = \nop f_2(x) =
		\begin{cases}
			[0, +\infty\mathclose[ & \text{if } x > 0,\\
			\R & \text{if } x = 0,\\
			\mathopen]-\infty, 0] & \text{if } x < 0.
		\end{cases}
	\]
	By \Cref{thm_characterization_maximal}, the operator $T \coloneqq \nop f_1 = \nop f_2$ is maximal cyclically quasi-monotone, and it admits two potentials whose families of sublevel sets, $\{[-t, t]\}_{t \geq 0}$ and $\{[-t, t/2]\}_{t \geq 0}$, differ. In particular, $f_2$ is not of the form $\varphi \circ f_1$ with $\varphi$ increasing.
\end{xmp}

We now have all the ingredients for the proof of \Cref{main_thm}.

\begin{proof}[Proof of \Cref{main_thm}]
	\textit{Item \labelcref{main_thm_existence}.} If $T \subseteq \nop f$ for some quasi-convex function $f$, then $T$ is cyclically quasi-monotone by \Cref{convex_chain_operator_CQM}. Conversely, assume that $T$ is cyclically quasi-monotone. Since $\X$ is separable, it is smoothable, and \Cref{thm_chain} provides a lower semicontinuous quasi-convex preference relation $\preceq$ such that $T \subseteq \nop_\preceq$. By \Cref{cor_rader_separable}, $\preceq$ is represented by some $f \in \cQ_0(\X)$, so that $\nop_\preceq = \nop f$, hence $T \subseteq \nop f$.

	\textit{Items \labelcref{main_thm_maximal_CQM} and \labelcref{main_thm_maximal_QM}} follow from \Cref{thm_characterization_maximal}, and \textit{item \labelcref{main_thm_nonuniqueness}} from \Cref{example_non_uniqueness}.
\end{proof}

\section{Applications}\label{sec_applications}
In this section, we present two applications of \Cref{main_thm}: the first one to revealed preference theory in economics (\Cref{sec_revealed_preferences}), the second one to supremal optimal transport (\Cref{sec_supremal_OT}).

\subsection{Revealed preferences}\label{sec_revealed_preferences}
In revealed preference theory, one observes the choices of a consumer facing bundles of goods at various prices, and asks whether these choices are rational, in the sense that they can be explained by the maximization of a single hidden utility function. A \emph{data set} is a set $\mathcal D \subseteq \X \times \X^*$ of pairs $(x,p)$, where $x$ is a chosen bundle of goods and $p$ the corresponding price vector (classically, $\X = \R^d$ and bundles and prices are nonnegative). The quantity $\pair{p}{x}$ is the price paid for $x$. If $\pair{p}{y} \leq \pair{p}{x}$, then the bundle $y$ was affordable when $x$ was chosen, so that $x$ is \emph{directly revealed preferred} to $y$. If moreover $\pair{p}{y} < \pair{p}{x}$, then $x$ is \emph{strictly} directly revealed preferred to $y$.

A utility function $u \colon \X \to \R$ \emph{rationalizes} $\mathcal D$ if every chosen bundle maximizes $u$ among the affordable ones, that is,
\begin{equation}\label{classical_rationalization}
	\forall (x,p) \in \mathcal D,\ \forall y \in \X, \qquad \pair{p}{y} \leq \pair{p}{x} \implies u(y) \leq u(x).
\end{equation}
Every data set is rationalized by a constant function, so that this notion is only meaningful under an additional requirement on $u$, such as local non-satiation. Following Reny \cite[footnote~2]{renyCharacterizationRationalizableConsumer2015}, one may instead require strictly cheaper bundles to be strictly worse, which we call \emph{rationalization in the strict sense}:
\begin{equation}\label{strict_rationalization}
	\forall (x,p) \in \mathcal D,\ \forall y \in \X, \qquad \pair{p}{y} < \pair{p}{x} \implies u(y) < u(x),
\end{equation}
which rules out such trivial rationalizations. If $u$ is locally non-satiated, then \eqref{classical_rationalization} implies \eqref{strict_rationalization}: if $\pair{p}{y} < \pair{p}{x}$, then every $z$ close enough to $y$ is still affordable, and choosing such a $z$ with $u(z) > u(y)$ gives $u(y) < u(z) \leq u(x)$. Conversely, if $u$ is lower semicontinuous and all prices are nonzero, then \eqref{strict_rationalization} implies \eqref{classical_rationalization}: if $\pair{p}{y} = \pair{p}{x}$, then $y$ is the limit of strictly cheaper bundles $y_n$, so that $u(y) \leq \liminf_n u(y_n) \leq u(x)$. Reny requires both \eqref{classical_rationalization} and \eqref{strict_rationalization}, while the rationalization in the strict sense of \Cref{cor_afriat} only requires \eqref{strict_rationalization}.

Several axioms on $\mathcal D$ express that the observed choices are consistent, in the sense that they contain no contradictory cycle of revealed preferences. Let $(x_0,p_0), \ldots, (x_{N-1},p_{N-1})$ be elements of $\mathcal D$, set $x_N \coloneqq x_0$, and assume that they form a cycle of direct revealed preferences:
\begin{equation}\label{revealed_preference_cycle}
	\pair{p_k}{x_{k+1}} \leq \pair{p_k}{x_k} \qquad \text{for every } k < N.
\end{equation}
The data set $\mathcal D$ satisfies
\begin{itemize}
	\item the \emph{Weak Axiom of Revealed Preference} (WARP) \cite{samuelsonNotePureTheory1938} if \eqref{revealed_preference_cycle} is impossible for $N = 2$ and $x_0 \neq x_1$,
	\item the \emph{Strong Axiom of Revealed Preference} (SARP) \cite{houthakkerRevealedPreferenceUtility1950} if \eqref{revealed_preference_cycle} is impossible with $x_0, \ldots, x_{N-1}$ not all equal,
	\item the \emph{Generalized Axiom of Revealed Preference} (GARP) \cite{varianNonparametricApproachDemand1982} if \eqref{revealed_preference_cycle} is impossible with at least one strict inequality.
\end{itemize}
For instance, WARP means that if $(x,p), (y,q) \in \mathcal D$ with $x \neq y$ and $\pair{p}{y} \leq \pair{p}{x}$, then $\pair{q}{x} > \pair{q}{y}$. Indeed, $x$ is then revealed preferred to $y$, and since $y$ was nevertheless chosen at prices $q$, $x$ was not affordable with the budget $\pair{q}{y}$. SARP extends this requirement to cycles of any length, which accounts for the transitivity of preferences. GARP allows for indifference: it only forbids the cycles containing a strict revealed preference, and a cycle is allowed if every bundle costs exactly the budget of the previous one. By Afriat's theorem (\Cref{thm_afriat}), GARP characterizes the finite data sets with nonnegative bundles and positive prices that are rationalized by a locally non-satiated utility function. Note that SARP implies both WARP and GARP, which are independent in general.

The link with our problem appears on the operator $T(x) \coloneqq \{-p : (x,p) \in \mathcal D\}$, the opposite of the price operator. With $x^*_k = -p_k \in T(x_k)$, the cycle \eqref{revealed_preference_cycle} reads $\pair{x^*_k}{x_{k+1}-x_k} \geq 0$ for every $k < N$. Therefore GARP means that $T$ is \emph{cyclically pseudo-monotone} \cite{daniilidisSubdifferentialsQuasiconvexPseudoconvex1999}: no cycle satisfies $\pair{x^*_k}{x_{k+1}-x_k} \geq 0$ for every $k$, with at least one strict inequality. On the other hand, the cyclical quasi-monotonicity of $T$ means that there is no cycle of strict direct revealed preferences, that is, no cycle \eqref{revealed_preference_cycle} in which all the inequalities are strict. This is the \emph{Quasi Axiom of Revealed Preference} (QARP) of \Cref{cor_afriat}, which is weaker than GARP. Moreover, rationalization in the strict sense is exactly our integration problem for $T$. Indeed, setting $f \coloneqq -u$, the contrapositive of \eqref{strict_rationalization} reads
\begin{equation}\label{rationalization_normal_cone}
	\forall (x,p) \in \mathcal D,\ \forall y \in \X, \qquad f(y) \leq f(x) \implies \pair{-p}{y-x} \leq 0,
\end{equation}
that is, $-p \in \nop f(x)$. Thus $u$ rationalizes $\mathcal D$ in the strict sense if and only if $T \subseteq \nop f$. Through this correspondence, \Cref{cor_afriat} shows that QARP characterizes the data sets rationalized in the strict sense, for arbitrary data sets in separable Banach spaces.
\begin{proof}[Proof of \Cref{cor_afriat}]
	(i) $\Rightarrow$ (ii). Let $u$ rationalize $\mathcal D$ in the strict sense. If elements $(x_k,p_k)_{k<N}$ of $\mathcal D$ formed a cycle of strict direct revealed preferences, then $u(x_{k+1}) < u(x_k)$ for every $k < N$ by \eqref{strict_rationalization}, hence $u(x_0) = u(x_N) < u(x_0)$, a contradiction.

	(ii) $\Rightarrow$ (i). By (ii), the operator $T(x) = \{-p : (x,p) \in \mathcal D\}$ is cyclically quasi-monotone. By item~\labelcref{main_thm_existence} of \Cref{main_thm}, $T \subseteq \nop f$ for some $f \in \cQ_0(\X)$, so that $u \coloneqq -f$ is upper semicontinuous and quasi-concave, and rationalizes $\mathcal D$ in the strict sense by \eqref{rationalization_normal_cone}. If $\X$ is finite-dimensional, $T$ admits a maximal cyclically quasi-monotone extension by \Cref{existence_maximal_extension}, which is of the form $\nop f$ with $f \in \cE_0(\X)$ by item~\labelcref{main_thm_maximal_CQM} of \Cref{main_thm}. Then $u \coloneqq -f$ is moreover essentially non-satiated.
\end{proof}

\subsection{Supremal optimal transport}\label{sec_supremal_OT}
Let $\mu$ and $\nu$ be Borel probability measures on $\R^d$ with compact supports, and let $\Pi(\mu,\nu)$ be the set of \emph{transport plans} between them, that is, Borel probability measures on $\R^d \times \R^d$ with marginals $\mu$ and $\nu$. Given a continuous cost function $c \colon \R^d \times \R^d \to \R$, the classical (Kantorovich) optimal transport problem reads
\begin{equation}\label{OT_problem}
	\min_{\gamma \in \Pi(\mu,\nu)} \int c \, d\gamma.
\end{equation}
Typical costs are the Euclidean distance $c(x,y) = \abs{x-y}$ and, more commonly, its square $c(x,y) = \abs{x-y}^2$. In the latter case, expanding the square shows that \eqref{OT_problem} has the same minimizers as with the cost $c(x,y) = -\pair{x}{y}$, since the terms $\abs{x}^2$ and $\abs{y}^2$ integrate to constants. Optimal plans are characterized by the \emph{$c$-cyclical monotonicity} of their supports \cite[Chapter~1]{santambrogioOptimalTransportApplied2015}: a plan $\gamma \in \Pi(\mu,\nu)$ is optimal if and only if for every $(x_0,y_0), \ldots, (x_{N-1},y_{N-1}) \in \supp \gamma$, setting $x_N \coloneqq x_0$,
\[
	\sum_{i=0}^{N-1} c(x_i,y_i) \leq \sum_{i=0}^{N-1} c(x_{i+1},y_i).
\]
For $c(x,y) = -\pair{x}{y}$, this means that $\supp \gamma$, seen as the graph of an operator $T_\gamma \colon \R^d \multimap \R^d$, is cyclically monotone. By Rockafellar's theorem (\Cref{Rockafellar_Theorem}), $T_\gamma \subseteq \partial\varphi$ for some proper lower semicontinuous convex function $\varphi$, which prescribes the target of every point. This is the key step of Brenier's theorem \cite{brenierPolarFactorizationMonotone1991}.

The \emph{supremal}, or $L^\infty$, optimal transport problem consists in minimizing the worst cost instead of the average one:
\begin{equation}\label{Linfty_OT_problem}
	\min_{\gamma \in \Pi(\mu,\nu)} C_\infty(\gamma), \qquad \text{where } C_\infty(\gamma) \coloneqq \gamma\text{-}\operatorname*{ess\,sup} c = \max_{\supp \gamma} c.
\end{equation}
For $c(x,y) = \abs{x-y}$, its value is the $\infty$-Wasserstein distance between $\mu$ and $\nu$ \cite{championInftyWassersteinDistance2008}. Solutions exist, but are far from unique: an optimal plan can be modified away from the region where the maximal cost is reached, as long as the cost does not exceed this maximum. As for supremal functionals in the calculus of variations, where Aronsson introduced absolute minimizers \cite{aronssonMinimizationProblemsFunctional1965}, one selects the minimizers which are also optimal locally. A plan $\gamma \in \Pi(\mu,\nu)$ is a \emph{restrictable solution} of \eqref{Linfty_OT_problem} if, for every nonzero Borel measure $\gamma' \leq \gamma$, the probability measure $\gamma'/\gamma'(\R^d \times \R^d)$ is a solution of \eqref{Linfty_OT_problem} between its own marginals \cite{championInftyWassersteinDistance2008}. Restrictable solutions are characterized by the \emph{$\infty$-cyclical monotonicity} of their supports: for every $(x_0,y_0), \ldots, (x_{N-1},y_{N-1}) \in \supp \gamma$, setting $x_N \coloneqq x_0$,
\begin{equation}\label{infinite_cyclical_monotonicity}
	\max_{0 \leq i < N} c(x_i,y_i) \leq \max_{0 \leq i < N} c(x_{i+1},y_i).
\end{equation}
This was shown by Champion, De Pascale and Juutinen for $c(x,y) = \abs{x-y}$ \cite{championInftyWassersteinDistance2008}, and by Jylhä for more general costs, in particular continuous ones \cite{jylhaInftyOptimalTransport2015}. Moreover, if $c$ is nonnegative, restrictable solutions can be obtained as limits of solutions of the $L^p$ problems $\min_{\gamma \in \Pi(\mu,\nu)} \int c^p \, d\gamma$ as $p \to \infty$ \cite{championInftyWassersteinDistance2008,jylhaInftyOptimalTransport2015}.

The reduction from $\abs{x-y}^2$ to $-\pair{x}{y}$ no longer works for \eqref{Linfty_OT_problem}, since $C_\infty$ is not linear in $\gamma$, so that these two costs lead to different problems. For $c(x,y) = -\pair{x}{y}$, $\infty$-cyclical monotonicity implies cyclical quasi-monotonicity, as observed in \cite{depascaleRockafellarTheoremCyclically2025}.

\begin{lem}\label{infinite_CM_implies_CQM}
	Let $c(x,y) = -\pair{x}{y}$ and let $\Gamma \subseteq \R^d \times \R^d$ satisfy \eqref{infinite_cyclical_monotonicity} for every $(x_0,y_0), \ldots, (x_{N-1},y_{N-1}) \in \Gamma$. Then the operator $x \mapsto \{y \in \R^d : (x,y) \in \Gamma\}$ is cyclically quasi-monotone.
\end{lem}

\begin{proof}
	Let $(x_0,y_0), \ldots, (x_{N-1},y_{N-1}) \in \Gamma$ and set $x_N \coloneqq x_0$. Changing signs in \eqref{infinite_cyclical_monotonicity}, and since the minimum of a sum is at least the sum of the minima, we get
	\[
		\min_{0 \leq i < N} \pair{y_i}{x_i} \geq \min_{0 \leq i < N} \bigl(\pair{y_i}{x_i} + \pair{y_i}{x_{i+1}-x_i}\bigr) \geq \min_{0 \leq i < N} \pair{y_i}{x_i} + \min_{0 \leq i < N} \pair{y_i}{x_{i+1}-x_i}.
	\]
	Hence $\min_i \pair{y_i}{x_{i+1}-x_i} \leq 0$, which is \eqref{cyclical_quasi_monotonicity}.
\end{proof}

We obtain that every restrictable solution admits a nontrivial quasi-convex potential, which prescribes the direction of the target of every point.

\begin{thm}\label{thm_supremal_OT}
	Let $\mu$ and $\nu$ be Borel probability measures on $\R^d$ with compact supports, and let $c(x,y) = -\pair{x}{y}$. Then the following statements hold.
	\begin{enumerate}[(i)]
		\item The supremal optimal transport problem \eqref{Linfty_OT_problem} admits a restrictable solution.
		\item For every restrictable solution $\gamma$ of \eqref{Linfty_OT_problem}, there exists an essentially non-satiated function $f \in \cE_0(\R^d)$ such that
		\[
			y \in \nop f(x) \qquad \text{for every } (x,y) \in \supp \gamma.
		\]
	\end{enumerate}
\end{thm}

\begin{proof}
	(i) Since $\mu$ and $\nu$ have compact supports, there is a constant $K$ such that $c + K \geq 0$ on $\supp \mu \times \supp \nu$, on which every plan $\gamma \in \Pi(\mu,\nu)$ is concentrated. Replacing $c$ by $c + K$ adds $K$ to $C_\infty(\gamma)$ for every plan $\gamma$, hence does not change the solutions of \eqref{Linfty_OT_problem}, its restrictable solutions, nor the condition \eqref{infinite_cyclical_monotonicity}. For every $p \geq 1$, let $\gamma_p$ be a solution of $\min_{\gamma \in \Pi(\mu,\nu)} \int (c+K)^p \, d\gamma$. Since $\Pi(\mu,\nu)$ is weakly compact, a subsequence of $(\gamma_p)$ converges weakly to some $\gamma \in \Pi(\mu,\nu)$ as $p \to \infty$, and $\gamma$ is a restrictable solution of \eqref{Linfty_OT_problem} \cite{championInftyWassersteinDistance2008,jylhaInftyOptimalTransport2015}.

	(ii) Let $\gamma$ be a restrictable solution of \eqref{Linfty_OT_problem}. Its support satisfies \eqref{infinite_cyclical_monotonicity} \cite{championInftyWassersteinDistance2008,jylhaInftyOptimalTransport2015}, so the operator $T_\gamma$ whose graph is $\supp \gamma$ is cyclically quasi-monotone by \Cref{infinite_CM_implies_CQM}. By \Cref{existence_maximal_extension} and item~\labelcref{main_thm_maximal_CQM} of \Cref{main_thm}, $T_\gamma \subseteq \nop f$ for some $f \in \cE_0(\R^d)$. Hence $y \in T_\gamma(x) \subseteq \nop f(x)$ for every $(x,y) \in \supp \gamma$.
\end{proof}

\section*{Acknowledgments}
Part of this work was carried out during visits of the first author to the CEREMADE laboratory of Université Paris Dauphine--PSL and the MOKAPLAN team of Inria Paris, and during a visit of the second author to the Dipartimento di Matematica e Informatica \enquote{Ulisse Dini} of the Università di Firenze in the fall of 2025. The authors gratefully acknowledge the warm hospitality and support of these institutions.

The work of the first author is partially funded by the \textit{Fondi di ricerca di ateneo} of the Università di Firenze and by the European Union -- NextGenerationEU (Missione 4, Componente 2, Investimento 1.1, \textit{Progetti di Ricerca di Rilevante Interesse Nazionale} (PRIN), CUP B53D23009310006, project 2022J4FYNJ). The first author is a member of the research group GNAMPA of INdAM.

Claude (Anthropic) was used to assist with the revision of a first draft of this manuscript, including literature search, editing, exposition and reviewing. The mathematical results and proofs are the authors' own.

\printbibliography
\end{document}